\documentclass[12pt]{amsart}
\usepackage{amssymb}
\usepackage{amsmath,amssymb,amsfonts,amsthm,
latexsym, amscd, epsfig, epic,color}
\usepackage{mathrsfs}
\usepackage{eucal}
\usepackage{eufrak}
\usepackage{xspace}
\usepackage{a4wide}
\usepackage{colordvi}
\usepackage{comment}
\usepackage{cite}
\usepackage{hyperref}

\usepackage{ytableau}
\usepackage[margin=1in]{geometry}
\usepackage{stmaryrd}
\usepackage{shadow}
\usepackage{graphicx}
\usepackage{longtable}
\usepackage{pstricks,multido}
\usepackage{qtree}
\usepackage{tikz}\usetikzlibrary{matrix,arrows}
\usepackage[utf8]{inputenc}
\usepackage[T1]{fontenc}
\allowdisplaybreaks

\newcounter{mycount}
\theoremstyle{plain}
\newtheorem{theorem}[mycount]{Theorem}
\newtheorem{corollary}[mycount]{Corollary}

\newtheorem{lemma}[mycount]{Lemma}
\newtheorem{proposition}[mycount]{Proposition}
\newtheorem{conjecture}[mycount]{Conjecture}
\newtheorem{remark}{Remark}
\theoremstyle{definition}
\newtheorem{definition}{Definition}
\newtheorem{fact}{Fact}

\theoremstyle{example}
\newtheorem{example}{Example}

\theoremstyle{remark}
\numberwithin{equation}{section}
\numberwithin{figure}{section}

\newcommand\des{\mathop{\rm des}}

\def\des{\mathsf{des}}
\def\ides{\mathsf{ides}}
\def\iasc{\mathsf{iasc}}
\newcommand{\asc}{\mathsf{asc}}
\def\max{\mathsf{max}}
\def\zero{\mathsf{zero}}
\def\rep{\mathsf{rep}}
\def\rmin{\mathsf{rmin}}
\def\lmax{\mathsf{lmax}}
\def\lmin{\mathsf{lmin}}
\def\rmax{\mathsf{rmax}}
\def\nasc{\mathsf{nasc}}

\def\maxx{\mathrm{max}}

\def\ealm{\mathsf{ealm}}
\def\zpair{\mathsf{zpair}}
\def\mpair{\mathsf{mpair}}
\def\zpos{\mathsf{zpos}}
\def\mpos{\mathsf{mpos}}

\def\A{\mathcal{A}}
\def\T{\CMcal{T}}
\def\I{\operatorname{{\bf I}}}
\def\CS{ \CMcal{S}}
\def\B{ \mathcal{B}}

\def\IDES{\mathrm{IDES}}
\def\DES{\mathrm{DES}}
\def\RMAX{\mathrm{RMAX}}
\def\LMAX{\mathrm{LMAX}}
\def\LMIN{\mathrm{LMIN}}
\def\DIST{\mathrm{DIST}}
\def\ASC{\mathrm{ASC}}
\def\RMIN{\mathrm{RMIN}}
\def\ZERO{\mathrm{ZERO}}
\def\MAX{\mathrm{MAX}}
\def\NASC{\mathrm{NASC}}

\def\ST{\mathrm{ST}}

\def\S{\mathfrak{S}}

\def\st{\mathsf{st}}

\def\boxit#1{\leavevmode\hbox{\vrule\vtop{\vbox{\kern.33333pt\hrule\kern1pt\hbox{\kern1pt\vbox{#1}\kern1pt}}\kern1pt\hrule}\vrule}}

\usepackage{collectbox}


\newcommand{\PATTERN}{
    \draw[step=1, xshift=14pt, yshift=14pt, \cfill, line cap=round] (0,0) grid (3,3);
    \draw[step=1, xshift=14pt, yshift=14pt, thick] (0,2) -- (3,2);
    \draw[step=1, xshift=14pt, yshift=14pt, thick] (1,0) -- (1,3);
    \foreach \x/\y in {1/1,2/3,3/2} \node[disc, fill=black] at (\x,\y) {};
}

\newcommand{\SEPATTERN}{
    \draw[step=1, xshift=14pt, yshift=14pt, \cfill, line cap=round] (0,0) grid (3,3);
    \draw[step=1, xshift=14pt, yshift=14pt, thick] (0,1) -- (3,1);
    \draw[step=1, xshift=14pt, yshift=14pt, thick] (1,0) -- (1,3);
    \foreach \x/\y in {1/2,2/3,3/1} \node[disc, fill=black] at (\x,\y) {};
}

\newcommand{\pattern}{\!\raisebox{-0.3em}{
  \begin{tikzpicture}[line width=0.7pt, scale=0.15]
    \tikzstyle{disc} = [circle,thin,draw=black, minimum size=1.7pt, inner sep=0pt ]
    \PATTERN
  \end{tikzpicture}}
}

\newcommand{\sepattern}{\!\raisebox{-0.3em}{
  \begin{tikzpicture}[line width=0.7pt, scale=0.15]
    \tikzstyle{disc} = [circle,thin,draw=black, minimum size=1.7pt, inner sep=0pt ]
    \SEPATTERN
  \end{tikzpicture}}
}

\newcommand{\cfill}{black!40}

\begin{document}

\title[Euler--Stirling statistics on ascent sequences]{A new decomposition of ascent sequences and Euler--Stirling statistics}

\author[S. Fu]{Shishuo Fu}
\address[Shishuo Fu]{College of Mathematics and Statistics, Chongqing University,  Chongqing 401331, P.R. China}
\email{fsshuo@cqu.edu.cn}

\author[E.Y. Jin]{Emma Yu Jin}
\address[Emma Yu Jin]{
Fakult\"at f\"ur Mathematik, Universit\"at Wien, 1090 Wien, Austria
}
\email{yu.jin@univie.ac.at}

\author[Z. Lin]{Zhicong Lin}
\address[Zhicong Lin]{Research Center for Mathematics and Interdisciplinary Sciences, Shandong University, Qingdao 266237, P.R. China}
\email{zhicong.lin@univie.ac.at}

\author[S.H.F. Yan]{Sherry H.F. Yan}
\address[Sherry H.F. Yan]{Department of Mathematics,
Zhejiang Normal University, Jinhua 321004, P.R. China}
\email{hfy@zjnu.cn}

\author[R.D.P. Zhou]{Robin D.P. Zhou}
\address[Robin D.P. Zhou]{College of Mathematics Physics and Information,
Shaoxing University, Shaoxing 312000, P.R. China}
\email{dapao2012@163.com}

\date{\today}

\begin{abstract}
As shown by Bousquet-M\'elou--Claesson--Dukes--Kitaev (2010),  ascent sequences can be used to encode $({\bf2+2})$-free posets. It is known that ascent sequences are enumerated by the Fishburn numbers, which appear as
the coefficients of the formal power series
$$\sum_{m=1}^{\infty}\prod_{i=1}^m (1-(1-t)^i).$$
In this paper, we present a novel way to recursively decompose ascent sequences, which leads to:
\begin{itemize}
\item a calculation of the Euler--Stirling distribution on ascent sequences, including the numbers of ascents ($\asc$), repeated entries $(\rep)$, zeros ($\zero$) and maximal entries ($\max$). In particular, this  confirms and extends Dukes and Parviainen's conjecture on the equidistribution of $\zero$ and $\max$.
\item a far-reaching generalization of the generating function formula for $(\asc,\zero)$ due to Jel\'inek. This is accomplished  via a bijective proof of the quadruple equidistribution of $(\asc,\rep,\zero,\max)$ and $(\rep,\asc,\rmin,\zero)$, where $\rmin$ denotes the right-to-left minima statistic of ascent sequences.
\item  an extension of a conjecture posed by Levande, which asserts that the pair $(\asc,\zero)$ on ascent sequences has the same distribution as the pair $(\rep,\max)$ on $({\bf2-1})$-avoiding inversion sequences.  This is achieved via a decomposition of $({\bf2-1})$-avoiding inversion sequences parallel to that of ascent sequences.
\end{itemize}

 This work is motivated by a double Eulerian equidistribution of Foata (1977) and a tempting bi-symmetry conjecture, which asserts that the  quadruples  $(\asc,\rep,\zero,\max)$ and $(\rep,\asc,\max,\zero)$ are equidistributed  on ascent sequences.
\end{abstract}

\keywords{ascent sequences, $({\bf2-1})$-avoiding inversion sequences, ascents, distinct entries, maximal entries, Fishburn numbers}

\maketitle


\section{Introduction}

Ascent sequences were introduced by Bousquet-M\'elou--Claesson--Dukes--Kitaev~\cite{bcdk}
to unify three seemingly unrelated  combinatorial structures: $({\bf2+2})$-free posets, a family of
permutations avoiding a certain pattern and Stoimenow's involutions~\cite{sto,zag}.
Many other  equinumerous objects, including Fishburn matrices~\cite{dp,je,yan}, $({\bf2-1})$-avoiding inversion sequences and non-$2$-neighbor-nesting matchings~\cite{cl,lev}, and various  interesting statistics on them~\cite{dp,dkrs,je2,kr2,kr} have been extensively investigated in these years. This paper is devoted to a systematic study of joint distributions of some classical word statistics, which we classify as {\em Eulerian} or {\em Stirling} statistics, on ascent sequences. In particular, two conjectures regarding ascent sequences due  respectively to Dukes--Parviainen~\cite{dp} and Levande~\cite{lev}  are solved. Our central contribution is the discovery of a new  decomposition of ascent sequences.

Let us first review some necessary definitions and state our original motivation. The set $\I_n$ of {\em inversion sequences}  of length $n$,
$$
\I_n:=\{s=(s_1,s_2,\ldots, s_n): 0\leq s_i<i\},
$$
serves as various kinds of codings (cf.~\cite{bv,fo,kl,leh}) for the set $\S_n$ of all permutations of $[n]:=\{1,2,\ldots,n\}$. For example, the map $\Theta:\S_n\rightarrow\I_n$ defined for $\pi=\pi_1\pi_2\cdots\pi_n\in\S_n$ as
 $$
 \Theta(\pi)=(s_1,s_2,\ldots,s_n),\quad\text{where $s_i:=\left|\{j: \text{$j<i$ and $\pi_j>\pi_i$}\}\right|$},
 $$
 is a natural coding, known as the {\em Lehmer code}, of $\S_n$. For any sequence $s\in\I_n$, let
\begin{align*}
\asc(s)&:=|\{i\in[n-1]: s_i<s_{i+1}\}|,\\
\rep(s)&:=n-\vert\{s_1,s_2,\ldots,s_n\}\vert,\\
\zero(s)&:=\vert\{i\in[n]: s_i=0\}\vert,\\
\max(s)&:=\vert\{i\in[n]:s_i=i-1\}\vert\quad\text{and}\\
\rmin(s)&:=|\{i\in[n]: s_i<s_j\text{ for all $j>i$}\}|
\end{align*}
be the numbers of {\bf asc}ents, {\bf rep}eated entries,  {\bf zero}s, {\bf maxi}mal entries (or maximal for short) and {\bf r}ight-to-left {\bf min}ima of $s$, respectively. For example, if $s=(0,1,2,0,4,2,2,7)\in\I_8$, then $\asc(s)=4$, $\rep(s)=3$, $\zero(s)=2$, $\max(s)=5$ and $\rmin(s)=3$.  These five statistics on inversion sequences can be classified into two categories as follows.

The classical {\em Eulerian polynomial} $A_n(t)$ may be defined as the descent polynomial over permutations
$A_n(t):=\sum_{\sigma\in\S_n}t^{\des(\sigma)}$, where $\des(\sigma):=|\{i\in[n-1]: \sigma_i>\sigma_{i+1}\}|$.
The  {\em signless Stirling numbers of the first kind} count permutations by their number of {\em {\bf\em l}eft-to-right {\bf\em max}ima}, namely
$$
(x)_n:=x(x+1)\cdots(x+n-1)=\sum_{\sigma\in\S_n}x^{\lmax(\sigma)},
$$
where $\lmax(\sigma):=|\{i\in[n]:\sigma_i>\sigma_j\text{ for all $j<i$}\}|$. Therefore, each statistic whose distribution  gives $A_n(t)$ (resp.~$(x)_n$) can be called an {\em Eulerian statistic} (resp.~a {\em Stirling statistic}). So $\des$ is an Eulerian statistic and $\lmax,\lmin,\rmax$ are all Stirling statistics over permutations, where  $\lmin$ and $\rmax$ are   the numbers of left-to-right minima  and right-to-left maxima, respectively. Since Lehmer code $\Theta$ transforms the quadruple $(\des,\lmax,\lmin,\rmax)$ on $\S_n$ to $(\asc,\zero,\max,\rmin)$ on $\I_n$, $\asc$ is an Eulerian statistic while $\zero,\max,\rmin$ are all Stirling statistics over inversion sequences. Interestingly, Dumont~\cite{du} showed that $\rep$ is also an Eulerian statistic over inversion sequences.

A sequence $s\in\I_n$ is called an {\em ascent sequence} if for all $2\leq i\leq n$, $s_i$ satisfies
\begin{align*}
s_i\leq \asc(s_1,s_2,\ldots, s_{i-1})+1.
\end{align*}
As shown by Bousquet-M\'elou--Claesson--Dukes--Kitaev~\cite{bcdk},
ascent sequences can be used to encode unlabeled $({\bf2+2})$-free posets (isomorphic to interval orders by the work of Fishburn~\cite{fi1,fi2}).  Both objects are proved to be enumerated by the {\em Fishburn numbers}, appearing as the sequence A022493 in the OEIS~\cite{oeis} with the elegant generating function
 \begin{equation}\label{ser:fish}
 \sum_{m\geq1}\prod_{i=1}^m(1-(1-t)^i)=t+2t^2+5t^3+15t^4+53t^5+217t^6+1014t^7+\cdots.
 \end{equation}
 Moreover, under their bijection, the word statistics $\asc$ and $\zero$ on ascent sequences  correspond to the {\em magnitude} and the number of {\em minimal} elements of $({\bf2+2})$-free posets. In this paper, any statistic whose distribution over  a member of the Fishburn family equals the distribution of $\asc$ (resp.~$\zero$) on ascent sequences is called an {\em Eulerian statistic} (resp.~a {\em Stirling statistic}). Since the formal power series~\eqref{ser:fish} is a non-D-finite series, to show that a statistic is Eulerian or Stirling in the Fishburn family may  still be hard even if  its generating function has been calculated.


Our inspiration to consider the joint distribution of $(\asc,\rep)$ on  ascent sequences stems from the fact that this pair is symmetric over inversion sequences, that is
\begin{equation}\label{inv:sym}
\sum_{s\in\I_n}
u^{\asc(s)}x^{\rep(s)}
=\sum_{s\in\I_n}
u^{\rep(s)}x^{\asc(s)},
\end{equation}
which follows from an amazing double Eulerian equidistribution due to Foata~\cite{fo}:
\begin{equation}\label{dou:fo}
\sum_{s\in\I_n}
u^{\asc(s)}x^{\rep(s)}=\sum_{\pi\in\S_n}u^{\des(\pi)}x^{\des(\pi^{-1})}.
\end{equation}
Now a natural question arises: does the symmetry in~\eqref{inv:sym} still hold when $\I_n$ is replaced by its subset $\A_n$, the set of all ascent sequences of length $n$? This leads us to the following bi-symmetric conjecture involving  double Euler--Stirling  statistics over $\A_n$.
\begin{conjecture}\label{conj1ref}
For $n\geq1$, we have the bi-symmetric quadruple equidistribution
\begin{align*}
\sum_{s\in\A_n}
u^{\asc(s)}x^{\rep(s)}z^{\zero(s)}q^{\max(s)}
=\sum_{s\in\A_n}
u^{\rep(s)}x^{\asc(s)}z^{\max(s)}q^{\zero(s)}.
\end{align*}
\end{conjecture}
This conjecture has been verified for  $n$ up to $10$ by using Maple. At the beginning, it is hard even to show that $\rep$ is Eulerian on ascent sequences.
Our first main result is a formula for the generating function $G(t;x,q,u,z)$ of ascent sequences, counted by the length (variable $t$), $\rep$ (variable $x$), $\max$ (variable $q$), $\asc$ (variable $u$) and $\zero$ (variable $z$):
\begin{align*}
G(t;x,q,u,z)=\sum_{n\geq1} \left(\sum_{s\in\A_n}
x^{\rep(s)}q^{\max(s)}u^{\asc(s)}z^{\zero(s)}\right)t^n.
\end{align*}
\begin{theorem}\label{T:gen}
The generating function $G(t;x,q,u,z)$ of ascent sequences is
\begin{align*}
G(t;x,q,u,z)&=\sum_{m\geq0}\frac{zqr x^m(1-qr)(1-r)^m(x+u-xu)}
{[x(1-u)+u(1-qr)(1-r)^m][x+u(1-x)(1-qr)(1-r)^m]}\\
&\quad\times \prod_{i=0}^{m-1}\frac{1+(zr-1)(1-qr)(1-r)^i}{x+u(1-x)(1-qr)(1-r)^i},
\end{align*}
where $r=t\,(x+u-xu)$.
\end{theorem}
This formula is obtained by developing a new decomposition of ascent sequences. It may be possible to prove Conjecture~\ref{conj1ref} from our formula in Theorem~\ref{T:gen} by showing
$G(t;x,q,u,z)=G(t;u,z,x,q)$, although we have not succeeded so far. On the other hand, by setting $u=x=1$ in $G(t;x,q,u,z)$, we get a nice symmetric expression for the generating function of $(\zero,\max)$ on ascent sequences.
\begin{corollary}\label{zero:max}The generating function $G(t;1,q,1,z)$ of $(\zero,\max)$ on ascent sequences is
\begin{align*}
\sum_{m\geq0}qzt\prod_{i=0}^{m-1}[1-(1-zt)(1-qt)(1-t)^i].
\end{align*}
Consequently,  the pair $(\zero,\max)$ is symmetric on $\A_n$.
\end{corollary}
\begin{remark}
The $q=1$ case of Corollary~\ref{zero:max} was first conjectured by Kitaev and Remmel~\cite{kr2}, and then proved independently by Jel\'inek~\cite{je}, Levande~\cite{lev} and Yan~\cite{yan} via the connection with Fishburn matrices or Fishburn diagrams. Our proof here is the first  direct approach based only on the decomposition of ascent sequences. The $z=1$ case implies that $\max$ is a Stirling statistic on $\A_n$, proving a conjecture by Dukes and Parviainen~\cite[Conj.~13]{dp}.
\end{remark}

Our second main result is a combinatorial bijection, combining a recent coding due to Baril and Vajnovszki~\cite{bv} and two new bijections involving two subsets of inversion sequences, on ascent sequences.

\begin{theorem}\label{bij:sym}
There is a bijection $\Upsilon:\A_n\rightarrow\A_n$ which transforms the quadruple
$$
(\asc,\rep,\zero,\max) \text{ to } (\rep,\asc,\rmin,\zero).
$$
Consequently, $(\asc,\rep)$ is symmetric on $\A_n$, $\rep$ is a new Eulerian statistic and $\rmin$ is a new Stirling statistic.
\end{theorem}

The following  formulae are immediate consequences of Theorems~\ref{T:gen} and~\ref{bij:sym}.
\begin{corollary}
The pairs $(\asc,\zero)$ and $(\rep,\max)$ are equidistributed on $\A_n$ and have the common  generating function:
\begin{align}\label{E:gfrepmaxi}
G(t;1,1,u,z)&=\sum_{m\geq0}u^m\prod_{i=0}^m \frac{1-(1-zt)(1-t)^i}{u+(1-u)(1-zt)(1-t)^i}\\
\label{E:gf2}&=\sum_{m\geq0}\frac{zt(1-t)^{m+1}}{1-u+u(1-t)^{m+1}}
\prod_{i=0}^{m-1}(1-(1-zt)(1-t)^{i+1}).
\end{align}
\end{corollary}
\begin{remark}
Note that~\eqref{E:gfrepmaxi} was derived by Jel\'{i}nek via decomposing the primitive Fishburn matrices~\cite{je}, while expression~\eqref{E:gf2} appears to be new.
\end{remark}

For a sequence $s\in\I_n$, we say that $s$ is {\em $({\bf2-1})$-avoiding} if there exists {\em no}  $i<j$ such that $s_i=s_j+1$. The $({\bf2-1})$-avoiding inversion sequences were used by Claesson and Linusson~\cite{cl} to encode the non-$2$-neighbor-nesting matchings. Denote by $\T_n$  the set of all $({\bf2-1})$-avoiding inversion sequences of length $n$. Levande~\cite{lev} showed that $\max$ is  Stirling  on $\T_n$ by introducing the Fishburn diagrams. At the end of his paper, he suggested that $\rep$ is an Eulerian statistic over $\T_n$. Our third main result answers his conjecture affirmatively.

\begin{theorem}\label{main:3}
The pair $(\rep,\max)$ is equidistributed  over $\A_n$ and $\T_n$.
Consequently,
\begin{align}\label{E:rema2co}
\sum_{s\in\A_n}
u^{\asc(s)}z^{\zero(s)}=\sum_{s\in\A_n}
u^{\rep(s)}z^{\max(s)}
=\sum_{s\in\T_n}
u^{\rep(s)}z^{\max(s)}.
\end{align}
\end{theorem}

The rest of this paper is organized as follows. In Section~\ref{sec:2}, we develop a new decomposition of ascent sequences and give a proof of Theorem~\ref{T:gen}. In  Section~\ref{sec:3}, we construct the bijection $\Upsilon$ for Theorem~\ref{bij:sym}. A generalization of Theorem~\ref{main:3}, involving three new marginal statistics on ascent sequences or $({\bf2-1})$-avoiding inversion sequences, is proved in Section~\ref{sec:4}.

\section{A decomposition of ascent sequences}\label{sec:2}

This section is devoted to the proof of Theorem~\ref{T:gen}. Note that for any ascent sequence, all its maximal entries must appear exactly in the initial  strictly increasing subsequence. This particular property inspires us to consider the following parameter, which is crucial for our decomposition of ascent sequences.

\begin{definition}For a sequence $s=(s_1,\ldots,s_n)\in\A_n$ with $\max(s)=p<n$, the {\bf e}ntry {\bf a}fter the {\bf l}ast {\bf m}aximal of $s$ is denoted by $\ealm(s)$, that is, $\ealm(s)=s_{p+1}$. For example, if $s=(0,1,2,3,2,4)$, then $\ealm(s)=2$. For convenience, we set $\ealm(0,1,\ldots,n-1)=0$.
\end{definition}

Let $\A$ be the set of all ascent sequences. Denote by $|s|$ the length of a sequence $s$. Clearly, we have
$$
\sum_{\substack{s\in\A\\ \vert s \vert=\max(s)}}t^{\vert s\vert}x^{\rep(s)}q^{\max(s)}u^{\asc(s)}z^{\zero(s)}
=\sum_{n\geq1}(qt)^n\,u^{n-1}z=\frac{qtz}{1-qtu}.
$$
Let us introduce
\begin{align*}
F(t;x,q,w,u,z)&:=\sum_{\substack{s\in\A\\ \vert s \vert>\max(s)}}t^{\vert s\vert}x^{\rep(s)}q^{\max(s)}w^{\ealm(s)}u^{\asc(s)}z^{\zero(s)},\\
G(t;x,q,w,u,z)&:=qtz(1-qtu)^{-1}+F(t;x,q,w,u,z)
\end{align*}
and
$$a_p(t;x,w,u,z):=[q^p]F(t;x,q,w,u,z).$$
We will establish a functional equation for $F(t;x,q,w,u,z)$ by dividing the set
$\A':=\{s\in\A:\vert s\vert> \max(s)\}$
 into the following disjoint subsets:
 \begin{align*}
 \CMcal{S}_{1}&:=\{s\in\A': \vert s\vert=\max(s)+1\},\\
 \CMcal{S}_{2}&:=\{s\in\A'\setminus\CMcal{S}_{1}: s_{\max(s)+1}\ge s_{\max(s)+2}\},\\
 \CMcal{S}_{3}&:=\{s\in\A'\setminus\CMcal{S}_{1}: s_{\max(s)+1}< s_{\max(s)+2},\max(s)\notin\{s_i: \max(s)+2\leq i\leq |s|\}\},\\
 \CMcal{S}_{4}&:=\{s\in\A'\setminus\CMcal{S}_{1}: s_{\max(s)+1}< s_{\max(s)+2},\max(s)\in\{s_i: \max(s)+2\leq i\leq |s|\}\}.
 \end{align*}
 The main idea is to reduce the counting of ascent sequences to those either with smaller length or with longer initial strictly increasing subsequence. In the following, we deal with each of the above subsets separately.  
\begin{lemma}\label{L:case1}
The generating function for $\CS_1$ is
\begin{align}\label{eq:case1}
\sum_{s\in\CS_1}t^{\vert s\vert}x^{\rep(s)}q^{\max(s)}w^{\ealm(s)}u^{\asc(s)}z^{\zero(s)}
=\frac{qxzt^2(z+qtuw-qtuzw)}{(1-qtu)(1-qtuw)}.
\end{align}
\end{lemma}
\begin{proof}
We assume that $\max(s)=p$. If $\vert s\vert=\max(s)+1$, then the ascent sequence $s$ must be $(0,1,2,\ldots,p-1,i)$, where $0\le i=\ealm(s)<p$. Consequently, the left-hand-side of~\eqref{eq:case1} is equal to
\begin{align*}
\sum_{p=1}^{\infty}t^{p+1}x q^p w^0 u^{p-1}z^2 +\sum_{p=2}^{\infty}t^{p+1} x q^p \left(\sum_{i=1}^{p-1}w^i\right)u^{p-1}z
=\frac{qxzt^2(z+qtuw-qtuzw)}{(1-qtu)(1-qtuw)},
\end{align*}
as desired.
\end{proof}


\begin{lemma}\label{L:case2}
The generating function for $\CS_2$ is
\begin{multline}\label{eq:case2}
 \sum_{s\in\CS_{2}}t^{\vert s\vert}x^{\rep(s)}q^{\max(s)}w^{\ealm(s)}u^{\asc(s)}z^{\zero(s)}
\\=\frac{tx}{1-w}(F(t;x,q,w,u,z)-F(t;x,qw,1,u,z))+tx(z-1)F(t;x,q,0,u,z).
\end{multline}
\end{lemma}
\begin{proof}
Any sequence $s\in\CS_2\cap\A_n$ with $\max(s)=p$ and $\ealm(s)=i$ has the form
$$
s=(0,1,\ldots,p-1,i,j,s_{p+3},\ldots,s_n)
$$
for some $j\leq i\leq p-1$. By removing $i$ from $s$ we obtain
$$
\hat{s}=(0,1,\ldots,p-1,j,s_{p+3},\ldots, s_{n}).
$$
Clearly, the mapping $s\mapsto (\hat{s},i)$ is a bijection between
$$
\{s\in\CS_2\cap\A_{n}:\ealm(s)=i\}\quad\text{and}\quad\{(s,i): s\in\A_{n-1}\cap\A',  \ealm(s)\le i\le\max(s)-1\}
$$
such that
$$
\asc(s)=\asc(\hat{s}), \rep(s)=\rep(\hat{s})+1, \max(s)=\max(\hat{s}), \zero(s)=\zero(\hat{s})+\chi(i=0).
$$
Here $\chi(\mathsf{S})$ equals $1$, if the statement $\mathsf{S}$ is true; and $0$, otherwise.

Recall that $a_p(t;x,w,u,z)$ is the coefficient of $q^p$ in $F(t;x,q,w,u,z)$. By the above bijection, in order to derive the coefficient of $q^p$ in the left-hand-side of~\eqref{eq:case2}, we need to do the substitution
$$
w^j\rightarrow \sum_{i=j}^{p-1}w^i=\frac{w^j}{1-w}-\frac{w^p}{1-w}\quad \text{for  $1\le j<p$}
$$
 and
$$
w^0\rightarrow z+\sum_{i=1}^{p-1}w^i=\frac{1-w^p}{1-w}+(z-1)
$$
in the generating function $(tx)a_p(t;x,w,u,z)$. After multiplying $q^p$ and summing over all $p$,  we get the generating function for $\CS_2$:
\begin{align*}
&\quad\sum_{p=1}^{\infty}\frac{q^ptx}{1-w}a_p(t;x,w,u,z)
-\sum_{p=1}^{\infty}\frac{q^ptx}{1-w}a_p(t;x,1,u,z)w^p+\sum_{p=1}^{\infty}
tx(z-1)a_p(t;x,0,u,z)q^p\\
&=\frac{tx}{1-w}F(t;x,q,w,u,z)-\frac{tx}{1-w}F(t;x,qw,1,u,z)+tx(z-1)F(t;x,q,0,u,z),
\end{align*}
as desired.
\end{proof}


\begin{lemma}\label{L:case3a}
The generating function for $\CS_3$ is
\begin{align}
&\,\quad \sum_{s\in\CS_3}t^{\vert s\vert}x^{\rep(s)}q^{\max(s)}w^{\ealm(s)}u^{\asc(s)}z^{\zero(s)}\label{eq:case3}\\
\nonumber&=tux\left(\frac{w+z-wz}{1-w}\right)F(t;x,q,1,u,z)-\frac{tux}{1-w}F(t;x,q,w,u,z)\\
\nonumber&\quad-tux(z-1)F(t;x,q,0,u,z).
\end{align}
\end{lemma}
\begin{proof}
Any sequence $s\in\CS_{3}\cap\A_n$ with $\max(s)=p$ and $\ealm(s)=i$ has the form
$$
s=(0,1,\ldots,p-1,i,j,s_{p+3},\ldots,s_n)
$$
 for some $i<j<p$ and $p$ does not appear in  $(j,s_{p+3},\ldots,s_n)$. Construct $\hat{s}'$ as
$$
 \hat{s}'=(0,1,\ldots,p-1,j,s'_{p+3},\ldots,s'_n),
$$
where $s'_k=s_k-\chi(s_k>p)$ for $k\geq p+3$. It can be checked routinely that the mapping $s\mapsto (\hat{s}',i)$ establishes a bijection between
$$
\{s\in\CS_3\cap\A_{n}:\ealm(s)=i\}\quad\text{and}\quad\{(s,i): s\in\A_{n-1}, i<\ealm(s)\}
$$
 satisfying
$
\asc(s)=\asc(\hat{s}')+1$, $\rep(s)=\rep(\hat{s}')+1$, $\max(s)=\max(\hat{s}')$ and $\zero(s)=\zero(\hat{s}')+\chi(i=0)$.

In terms of the generating function, we have to do the substitution
$$
\quad w^j\rightarrow z+w+\cdots+w^{j-1}=\frac{1-w^j}{1-w}+(z-1)\quad \text{for  $1\le j<p$}
$$
and $w^0\rightarrow 0$ in $(tux)a_p(t;x,w,u,z)$. After multiplying $q^p$ and summing over all $p$,  we obtain the generating function for $\CS_3$:
\begin{align*}
&\quad\sum_{p=1}^{\infty}\frac{q^ptux}{1-w}(a_p(t;x,1,u,z)-a_p(t;x,w,u,z))\\
&\quad\quad+\sum_{p=1}^{\infty}tux(z-1)(a_p(t;x,1,u,z)-a_p(t;x,0,u,z))q^p\\
&=tux\left(\frac{w+z-wz}{1-w}\right)F(t;x,q,1,u,z)\\
&\quad\quad-\frac{tux}{1-w}F(t;x,q,w,u,z)-tux(z-1)F(t;x,q,0,u,z),
\end{align*}
as desired.
\end{proof}

The case for $\CS_4$ is a bit intricate. Let $\CMcal{P}$ be the set of all sequences $s\in\A$ such that the integer $\max(s)-1$ appears exactly once in $s$.
We treat the subset $\CS_4$ in the next two lemmas, which are also useful in proving Theorem~\ref{main:3}.

\begin{lemma}\label{L:case3}
There is a bijection $\phi:\A_n\cap\CMcal{P} \rightarrow \A_{n-1}$ that transforms the quintuple
$$
(\asc,\rep,\ealm,\max,\zero)\quad\text{to}\quad(\asc+1,\rep,\ealm,\max+1,\zero).
$$
\end{lemma}
\begin{proof}
Set $\phi(0,1,\ldots,n-1)=(0,1,\ldots,n-2)$. Any sequence $s\in(\A_n\cap\CMcal{P})\setminus\{(0,1,\ldots,n-1)\}$ with $\max(s)=p+1$ and $\ealm(s)=i$ has the form
$$
s=(0,1,\ldots,p-1,p,i,s_{p+3},\ldots,s_n),
$$
 where $p$ does not appear in  $(i,s_{p+3},\ldots,s_n)$. Define
 $$
 \phi(s)=(0,1,\ldots,p-1,i,s'_{p+3},\ldots,s'_n),
 $$
 where $s'_k=s_k-\chi(s_k>p)$ for $k\geq p+3$. It is easy to check that $\phi$ is a bijection with the desired properties. 
\end{proof}

\begin{lemma}\label{L:c3}
Let  $\CMcal{P}^c=\A\setminus\CMcal{P}$. There is a bijection
$$
\xi: \{s\in\A_n\cap\CS_4:\ealm(s)=i\} \rightarrow \{(s,i): s\in\A_n\cap\CMcal{P}^c,i<\ealm(s)\}
$$
such that if $\xi(s)=(s^*,i)$,
then $\asc(s)=\asc(s^*)$, $\rep(s)=\rep(s^*)$, $\max(s)=\max(s^*)-1$ and $\zero(s)=\zero(s^*)+\chi(i=0)$.
Furthermore, the generating function for $\CS_4$ is
\begin{align}
&\,\quad \sum_{s\in\CMcal{S}_4}t^{\vert s\vert}x^{\mathrm{rep}(s)}q^{\max(s)}w^{\ealm(s)}u^{\mathrm{asc}(s)}z^{\mathrm{zero}(s)}\label{eq:case4}\\
\nonumber&=\frac{(1-qut)(w+z-wz)}{q(1-w)}F(t;x,q,1,u,z)-\frac{1-qut}{q(1-w)}F(t;x,q,w,u,z)\\
\nonumber&\quad -\biggl(\frac{1}{q}-ut\biggr)(z-1)F(t;x,q,0,u,z).
\end{align}
\end{lemma}
\begin{proof}
Any sequence  $s\in \A_n\cap\CMcal{S}_4$ with $\max(s)=p$ and $\ealm(s)=i$ has the form
$$
s=(0,1,2,\ldots,p-1,i,j,s_{p+3},\ldots, s_{n})
$$
where $i<j\le p$ and the subsequence $(j,s_{p+3},\ldots,s_n)$ contains $p$.
Define  $\xi(s)=(s^*,i)$, where
$$s^*:=(0,1,2,\ldots,p-1,p,j,s_{p+3},\ldots, s_{n})\in \A_n\cap\CMcal{P}^c.
$$ It is easy to verify that $\xi$ is reversible and has the required properties.

Since $\CMcal{P}^c=\A\setminus\CMcal{P}$,  it follows from Lemma~\ref{L:case3} that  the length generating function  for $\{s\in\CMcal{P}^c: \max(s)=p+1\}$ by the weight $(\rep,\ealm,\asc,\zero)$ is
\begin{equation}\label{gen:pc}
a_{p+1}(t;x,w,u,z)-(tu)a_p(t;x,w,u,z).
\end{equation}
In view  of the bijection $\xi$, in order to get the coefficient of $q^p$ in the left-hand-side of~\eqref{eq:case4}, we need to do the substitution
$$
\quad w^j\rightarrow z+w+\cdots+w^{j-1}=\frac{1-w^j}{1-w}+(z-1)\quad \text{for  $1\le j\le p$}
$$
and $w^0\rightarrow 0$
in~\eqref{gen:pc}.
After multiplying $q^p$ and summing over all $p$,  we get the generating function for $\CS_4$:
\begin{align*}
&\quad \sum_{p=1}^{\infty}\left[\frac{a_{p+1}(t;x,1,u,z)}{1-w}-\frac{a_{p+1}(t;x,w,u,z)}{1-w}
+(z-1)(a_{p+1}(t;x,1,u,z)-a_{p+1}(t;x,0,u,z))\right]q^p\\
&\quad -\sum_{p=1}^{\infty}\left[\frac{uta_p(t;x,1,u,z)}{1-w}
-\frac{uta_p(t;x,w,u,z)}{1-w}+ut(z-1)(a_{p}(t;x,1,u,z)-a_{p}(t;x,0,u,z))\right]q^p\\
&=\frac{(1-qut)(w+z-wz)}{q(1-w)}F(t;x,q,1,u,z)-\frac{1-qut}{q(1-w)}F(t;x,q,w,u,z)\\
&\quad -\biggl(\frac{1}{q}-ut\biggr)(z-1)F(t;x,q,0,u,z),
\end{align*}
as desired.
\end{proof}
We are now in a position to complete the proof of Theorem~\ref{T:gen}. Combining equalities~\eqref{eq:case1}, \eqref{eq:case2}, \eqref{eq:case3} and~\eqref{eq:case4} we find that
\begin{align*}
\left(1-\frac{qt(x+u-ux)-1}{q(1-w)}\right)F(t;x,q,w,u,z)
&=\frac{qxzt^2(z+qtuw-qtuzw)}{(1-qtu)(1-qtuw)}\\
&\quad-\frac{tx}{1-w}F(t;x,qw,1,u,z)\\
&\quad+(z-1)(t(x+u-xu)-q^{-1})F(t;x,q,0,u,z)\\
&\quad+\frac{w+z-wz}{1-w}(uxt+q^{-1}-ut)F(t;x,q,1,u,z),
\end{align*}
which is equivalent to
\begin{align}\label{E:G2}
\nonumber\left(1-\frac{qt(x+u-ux)-1}{q(1-w)}\right)G(t;x,q,w,u,z)
\nonumber&=\frac{zqt}{1-qtu}-\frac{tx}{1-w}G(t;x,qw,1,u,z)\\
\nonumber&\quad+(z-1)(t(x+u-xu)-q^{-1})G(t;x,q,0,u,z)\\
&\quad+\frac{w+z-wz}{1-w}(uxt+q^{-1}-ut)G(t;x,q,1,u,z).
\end{align}
By setting $w=0$ and $r=t(x+u-xu)$, we find that
\begin{align*}
\left(1-r+q^{-1}\right)G(t;x,q,0,u,z)
&=\frac{zqt}{1-qtu}+(z-1)(r-q^{-1})G(t;x,q,0,u,z)\\
&\quad+z(uxt+q^{-1}-ut)G(t;x,q,1,u,z),
\end{align*}
which leads to
\begin{align}\label{E:Gw0}
G(t;x,q,0,u,z)=\frac{zqt}{(1-qtu)(1-zr+zq^{-1})}
+\frac{ztx-zr+zq^{-1}}{1-zr+zq^{-1}}G(t;x,q,1,u,z).
\end{align}
Plugging~\eqref{E:Gw0} into~\eqref{E:G2} we get
\begin{align*}
\left(1-\frac{qr-1}{q(1-w)}\right)&G(t;x,q,w,u,z)
=\frac{zqt}{(1-qtu)}\frac{(1-r+q^{-1})}{(1-zr+zq^{-1})}-\frac{tx}{1-w}G(t;x,qw,1,u,z)\\
&+\left(\frac{w+z-wz}{1-w}+\frac{z(z-1)(qr-1)}{q-qzr+z}\right)(tx-r+q^{-1})G(t;x,q,1,u,z).
\end{align*}
We use the kernel method to solve this equation. We set $w=1+q^{-1}-r$ to make the left-hand-side become zero  and as a result it can be further simplified into
\begin{align*}
G(t;x,q,1,u,z)&=\frac{zt}{(1-qtu)}\frac{1-qr}{(tx-r+q^{-1})}
+\frac{tx}{(tx-r+q^{-1})}\frac{1-zr+zq^{-1}}{1-r+q^{-1}}G(t;x,qw,1,u,z).
\end{align*}
Next we define $\delta_m=r^{-1}-r^{-1}(1-qr)(1-r)^m$ and clearly $\delta_1=q+1-qr$. We further define that
\begin{align*}
\mu_m&=\frac{z(1-qr)(1-r)^m(x+u-xu)[1-(1-qr)(1-r)^m]}
{[x(1-u)+u(1-qr)(1-r)^m][x+u(1-x)(1-qr)(1-r)^m]},\\
\nu_m&=\frac{x[1-(1-qr)(1-r)^m]}{[x+u(1-x)(1-qr)(1-r)^m]}
\frac{[1+(zr-1)(1-qr)(1-r)^m]}{[1-(1-qr)(1-r)^{m+1}]}.
\end{align*}
By iterating the previous equation, we have
\begin{align*}
&\quad G(t;x,q,1,u,z)=\sum_{m=0}^n\mu_m\prod_{i=0}^{m-1}\nu_i+\prod_{i=0}^{n}\nu_i \,G(t;x,\delta_{m+1},1,u,z).
\end{align*}
Consequently, as a power series in $t$, we conclude that $G(t;x,q,1,u,z)=\sum_{m=0}^{\infty}\mu_m\prod_{i=0}^{m-1}\nu_i$, which is
\begin{align*}
G(t;x,q,1,u,z)&=\sum_{m=0}^{\infty}\frac{x^m z(1-qr)(1-r)^m(x+u-xu)[1-(1-qr)(1-r)^m]}
{[x(1-u)+u(1-qr)(1-r)^m][x+u(1-x)(1-qr)(1-r)^m]}\\
&\quad\times \prod_{i=0}^{m-1}\frac{[1-(1-qr)(1-r)^i]}{[x+u(1-x)(1-qr)(1-r)^i]}
\frac{[1+(zr-1)(1-qr)(1-r)^i]}{[1-(1-qr)(1-r)^{i+1}]}\\
&=\sum_{m=0}^{\infty}\frac{zqr x^m(1-qr)(1-r)^m(x+u-xu)}
{[x(1-u)+u(1-qr)(1-r)^m][x+u(1-x)(1-qr)(1-r)^m]}\\
&\quad\times \prod_{i=0}^{m-1}\frac{1+(zr-1)(1-qr)(1-r)^i}{x+u(1-x)(1-qr)(1-r)^i},
\end{align*}
where $r=t\,(x+u-xu)$. The proof of Theorem~\ref{T:gen} is complete.



\section{The construction of $\Upsilon$} \label{sec:3}

Recently, restricted versions of Foata's double Eulerian equidistribution~\eqref{dou:fo} have been found by Kim and Lin~\cite{kl,kl2} for restricted permutations  enumerated by the {\em Schr\"oder numbers} and the {\em Euler numbers}. In this section, we will prove two restricted versions of~\eqref{dou:fo} for the Fishburn numbers involving two families of permutations avoiding generalized patterns (known as bivincular patterns) introduced in~\cite{bcdk}.

We say that a permutation $\pi\in\S_n$ contains the pattern \mbox{\sepattern} if there is a subsequence
$\pi_i\pi_{i+1}\pi_j$  of $\pi$ satisfying that $\pi_i-1=\pi_j $ and $\pi_i<\pi_{i+1}$,
otherwise we say that $\pi$ avoids the pattern  \mbox{\sepattern}.
For example, the permutation $1{\bf35}4{\bf2}$ contains the pattern \mbox{\sepattern},
while the permutation $31524$ avoids it.
Let
 $\S_n(\sepattern\,)$ be the set of
$(\sepattern\,)$-avoiding permutations in $\S_n$.
It was shown in \cite{bcdk} that $\S_n(\sepattern\,)$ is in bijection with $\A_n$ and thus  is enumerated by the $n$-th Fishburn number.

Let us recall some set-valued statistics on permutations and inversion sequences introduced in~\cite{bv,kl}. Let $\pi\in\S_n$ be a permutation. Define the positions of {\em {\bf\em des}cents} and  {\em {\bf\em i}nverse {\bf\em des}cents} of $\pi$ by
$$
\DES(\pi):=\{i\in[n-1]: \pi_i>\pi_{i+1}\}\quad\text{and}
$$
$$
\IDES(\pi):=\{2\leq i\leq n: \pi_i+1\text{ appears to the left of $\pi_i$}\}.
$$
The positions of {\em {\bf\em l}eft-to-right {\bf\em max}ima} of $\pi$ is $\LMAX(\pi):=\{i\in[n]:\pi_i>\pi_j\,\, \text{for all $1\leq j<i$}\}$. Similarly, we can define the  positions of {\em{\bf\em l}eft-to-right {\bf\em min}ima} $\LMIN(\pi)$ and  the positions of {\em {\bf\em r}ight-to-left {\bf\em max}ima} $\RMAX(\pi)$  of $\pi$.
Let $s\in\I_n$ be an inversion sequence. The positions of {\em {\bf\em asc}ents} and the {\em last occurrence of {\bf\em  dist}inct positive entries} of $s$ are
$$
\ASC(s):=\{i\in[n-1]: s_i<s_{i+1}\}\quad\text{and}
$$
$$
\DIST(s):=\{2\leq i\leq n: s_i\neq0\text{ and $s_i\neq s_j$ for all $j>i$}\}.
$$
The positions of {\em {\bf\em zero}s} in $s$ is $\ZERO(s):=\{i\in[n]: s_i=0\}$. The positions of the {\em entries of $s$ that achieve {\bf\em max}imum} is $\MAX(s):=\{i\in[n]: s_i=i-1\}$ and the positions of  {\em {\bf\em r}ight-to-left {\bf\em min}ima}  of $s$ is $\RMIN(s):=\{i\in[n]: s_i<s_j\text{ for all $j>i$}\}$. For convenience, we will use the convention that if the upper case ``$\ST$'' is a set-valued statistic, then the lower case ``$\st$'' is the corresponding numerical statistic.
Our first restricted double Eulerian equidistribution, involving the above set-valued statistics, is the following.

 \begin{theorem}\label{mainth}
 There is  a bijection $\Psi: \S_n(\sepattern\,)\rightarrow\mathcal{A}_n$  such  that  for any $\pi\in\S_n(\sepattern\,)$,
 $$
 (\DES, \IDES, \LMIN,  \LMAX, \RMAX)\pi=(\ASC, \DIST, \MAX,  \ZERO, \RMIN)\Psi(\pi).
 $$
 \end{theorem}

The bijection $\Psi$ will be a combination of a recent coding $b: \S_n\rightarrow\I_n$ due to Baril and Vajnovszki~\cite{bv} and a  bijection between $\A_n$ and a new subset of $\I_n$ that we introduce below.

For a sequence  $s\in\I_n$, define the positions of {\em {\bf\em n}on-{\bf\em asc}ents} of $s$ by
 $$
 \NASC(s):=\{i\in[n-1]: s_i\geq s_{i+1}\}=[n-1]\setminus\ASC(s).
 $$
Consider  the set $\B_n$  of  inversion sequences $b=(b_1,b_2,\ldots,b_n)\in\I_n$  satisfying the  following properties:
\begin{itemize}
\item[(a)]   if $b_i\geq b_{i+1}$ and $b_i=i-1$, then $b_j< j-1$ for all $j>i$;
\item[(b)]  if $b_i\geq b_{i+1}$ and $b_i<i-1$, then $i-1$ does not occur in  $(b_{i+1},b_{i+2},\ldots, b_{n})$.
\end{itemize}
For example, $(0,0,0,2)\in\B_4$ but $(0,0,0,1)\notin\B_4$.
We  aim to describe a map $\beta$ from $\mathcal{B}_n$ to $\mathcal{A}_n$ in   terms of an algorithm, which is similar in nature to the way modified ascent sequences in~\cite{bcdk} were created.   Let $b\in\B_n$ be an input sequence and suppose that $\NASC(b)=\{i_1, i_2, \ldots , i_k\}$ with $i_1>i_2>\cdots > i_k$. Do \\
\hspace*{0.5cm}  for $i=i_1,i_2,\ldots, i_k$:\\
\hspace*{1cm}  for $j=i+1, i+2, \ldots, n$:\\
 \hspace*{1.5cm}  if $b_i<i-1$ and $b_j>i-1$, then $b_j:=b_j-1$\\
and denote the resulting sequence by $\beta(b)$.
\begin{example}
Consider the sequence $b=(0,1,0,2,3,2,5,1,7)\in\B_9$. We have $\NASC(b)=\{7,5,2\}$  and the algorithm computes $\beta(b)$ in the following steps:
$$
\begin{array}{ccccccccccc}
b=&0&1&0&2&3&2&5&1&\bf{7}&\\
&0&1&0&2&3&2&\bf{5}&1&\bf{6}&\\
&0&1&0&2&3&2&4&1&5&\\
&0&1&0&2&3&2&4&1&5&=\beta(b).\\
\end{array}
$$
In each step every   boldfaced letter is decreased by
  one.
\end{example}

The property (b) ensures the following fact of the algorithm $\beta$.
\begin{fact}\label{beta:fact}
The relative order of entries in $b\in\B_n$ and $\beta(b)$ are the same.
\end{fact}

This fact implies the positions of non-ascents in $b$ and $\beta(b)$ coincide.   The above procedure is easy to invert by the following {\em addition} algorithm. Let $\beta(b)=(s_1,\ldots,s_n)$ be the input sequence. Do\\
\hspace*{0.5cm}  for $i=i_k, i_{k-1}\ldots, i_1$:\\
\hspace*{1cm}  for $j=i+1, i+2, \ldots, n$:\\
 \hspace*{1.5cm}  if $s_i<i-1$ and $s_j\geq i-1$, then $s_j:=s_j+1$.\\
Thus the map $\beta$ is injective.
\begin{proposition}\label{th2.1}
The map $\beta:\mathcal{B}_n\rightarrow\mathcal{A}_n$ is a bijection that preserves the quintuple of set-valued statistics
$(\ASC,\DIST,\MAX,\ZERO,\RMIN)$.
\end{proposition}

\begin{proof}
 First we shall prove that the map $\beta$ is well defined, that is, for any $b\in\B_n$ we have $\beta(b)=(s_1,s_2,\ldots, s_n)\in \A_n$. From    properties (a) and $(b)$ and the definition of $\beta$, we see that $s_i\leq i-1-\nasc(s_1,s_2,\ldots,s_{i})=\asc(s_1,s_2,\ldots, s_i)\leq \asc(s_1,s_2,\ldots, s_{i-1})+1$. This implies that the resulting sequence $\beta(b)$ is an ascent sequence.

 Recall that we have shown that the map $\beta$  is injective. In order to show that the map $\beta$ is a bijection, it remains to show that $\beta$ is surjective.  For any $s\in \A_n$, we can  get a sequence $b=(b_1,b_2,\ldots, b_n)$ by applying the addition algorithm.
 It is not difficult to check that the resulting sequence verifies properties $(a)$ and $(b)$. Moreover, we have $b_i\leq s_i+\nasc(s_1,s_2\ldots, s_{i})\leq \asc(s_1,s_2,\ldots,s_i)+\nasc(s_1,s_2,\ldots, s_{i})=i-1 $. This yields that   $b\in \B_n$, and thus $\beta$ is a bijection.

 Note that for any $b\in\B_n$, property (a) ensures that all  maximal entries of $b$ must appear exactly in the initial strictly increasing subsequence. Thus, $\beta$ preserves the statistic $\MAX$. It follows from Fact~\ref{beta:fact} that $\beta$ also preserves the other four statistics, which completes the proof.
 \end{proof}

Next we give an overview of the permutation code $b: \S_n\rightarrow\I_n$ from~\cite{bv}.
An interval $I=[p,q]$, $p\leq q$, is the set of integers $\{x: p\leq x\leq q\}$; and a labeled interval is  a pair $(I,\ell)$ where $I$ is an interval and $\ell$ is an integer.
For a permutation $\pi=\pi_1\pi_2\cdots \pi_n\in \S_n $ and an integer $i$,  $0\leq i<n$, the $i$th slice of $\pi$ is a sequence of labeled intervals
$U_i(\pi)=(I_1, \ell_1),(I_2, \ell_2),\ldots, (I_k, \ell_k)$ constructed recursively by the following procedure.

Set $U_0(\pi)=([0,n], 0)$. Assume that $U_{i-1}(\pi)=(I_1, \ell_1),(I_2, \ell_2),\ldots, (I_k, \ell_k)$ has been defined for some
$1\leq i<n$ and $\pi_i\in I_v$.  Below is a procedure to determine $U_{i}(\pi)$ from $U_{i-1}(\pi)$ according to four possible cases:
 \begin{itemize}
 \item If  $min(I_v)<\pi_i<max(I_v)$, then $U_i(\pi)$ equals
 $$
 (I_1, \ell_1), \ldots, (I_{v-1}, \ell_{v-1}),(H, \ell_v), (J, \ell_{v+1}), (I_{v+1}, \ell_{v+2}),\ldots, (I_{k-1}, \ell_k), (I_{k}, \ell_k+1),
 $$
 where $H=[\pi_{i}+1, max(I_v)]$  and $J=[min(I_v), \pi_i-1 ]$.
 \item If  $min(I_v)<\pi_i=max(I_v)$, then $U_i(\pi)$ equals
 $$
(I_1, \ell_1), \ldots, (I_{v-1}, \ell_{v-1}),  (J, \ell_{v+1}), (I_{v+1}, \ell_{v+2}), \ldots, (I_{k-1}, \ell_{k}), (I_{k }, \ell_k+1),
 $$
 where $J=[min(I_v), \pi_i-1 ]$.
 \item If  $min(I_v)=\pi_i<max(I_v)$, then $U_i(\pi)$ equals
 $$
 (I_1, \ell_1), \ldots, (I_{v-1}, \ell_{v-1}),  (H, \ell_{v})(I_{v+1}, \ell_{v+1}), \ldots, (I_{k-1}, \ell_{k-1}), (I_{k }, \ell_k+1),
 $$
 where $H=[\pi_{i}+1, max(I_v)]$.

 \item If  $min(I_v)=\pi_i=max(I_v)$, then  $U_i(\pi)$ equals
 $$
 (I_1, \ell_1), \ldots, (I_{v-1}, \ell_{v-1}),  (I_{v+1}, \ell_{v+1}),\ldots,  (I_{k-1}, \ell_{k-1}), (I_{k }, \ell_k+1).
 $$
 \end{itemize}
 Now the permutation code $b(\pi)=(b_1,b_2,\ldots, b_n)$ is defined by letting $b_i=\ell_v$ such that $(I_v, \ell_v)$ is a labeled interval in the $(i-1)$-th slice of $\pi$ with $\pi_i\in I_v$.

\begin{example}
Let $\pi=61832547$. The process described above gives the slices below.
$$
\begin{array}{lll}
U_0(\pi)&=& ([0,8], 0)\\
U_1(\pi)&=& ([7,8], 0)([0,5], 1)\\
U_2(\pi)&=& ([7,8], 0)([2,5], 1)([0,0], 2) \\
U_3(\pi)&=&([7,7], 1)([2,5], 2)([0,0], 3)\\
U_4(\pi)&=&([7,7], 1)([4,5], 2)([2,2], 3)([0,0], 4)\\
U_5(\pi)&=& ([7,7], 1)([4,5], 2)([0,0], 5)\\
U_6(\pi)&=& ([7,7], 1)([4,4], 5)([0,0], 6)\\
U_7(\pi)&=&([7,7], 1)([0,0], 7)  \\
\end{array}
$$
Thus, we have $b(\pi)=(0,1,0,2,3,2,5,1)$.
\end{example}

The permutation code $b$ constructed above proves a different set-valued extension of Foata's double Eulerian equidistribution.
\begin{theorem}{ \upshape   (Baril and Vajnovszki~\cite{bv})}\label{th2.3}
The code $b: \S_n\rightarrow\I_n$ is a bijection  such  that  for any $\pi\in\S_n$,
 $$
 (\DES, \IDES, \LMIN,  \LMAX, \RMAX)\pi=(\ASC, \DIST, \MAX,  \ZERO, \RMIN)b(\pi).
 $$
\end{theorem}

\begin{lemma}\label{th2.4}
 The permutation code $b$ induces a bijection between  $\S_n(\sepattern\,)$ and $\B_n$.
\end{lemma}
\begin{proof}
 By Proposition~\ref{th2.1}, the sets $\S_n(\sepattern\,)$ and $\mathcal{B}_n$ have the same cardinality. In order to prove that  $b$ induces a bijection between  $\S_n(\sepattern\,)$ and $\mathcal{B}_n$,  it remains to show that  for any $\pi\in\S_n(\sepattern\,)$, we have $b(\pi)=(b_1,b_2,\ldots, b_n)\in \mathcal{B}_n$.
To this end, it suffices to verify that $b(\pi)$ has properties $(a)$ and $(b)$.
Suppose that  $(I_v, \ell_v )$ is a labeled interval in the $(i-1)$th slice of $\pi$ with $\pi_i\in I_v $.  Clearly, we have $b_i=\ell_v$. Suppose that $b_i\geq b_{i+1}$.  By Theorem~\ref{th2.3}, $\DES(\pi)=\ASC(b(\pi))$ and so $\pi_i< \pi_{i+1}$. Since $\pi$ avoids the pattern $\sepattern\,$, if $\pi_i>1$, then the letter $\pi_i-1$ must appear before $\pi_i$ in $\pi$ and consequently  $min(I_v)=\pi_i$. From the definition of the $i$-th slice of $\pi$, we see that  if $b_i<i-1$, then  the element $i-1$ will never occur in the suffix $b_{i+1}b_{i+2}\ldots b_n$;  and if $b_i=i-1$, we have $b_j<j-1$ for all $j>i$. Thus, the resulting sequence $b(\pi)$ has properties $(a)$ and $(b)$, completing the proof.
\end{proof}

\begin{proof}[{\bf Proof of Theorem \ref{mainth}}] Let $\Psi=\beta \circ b$.  It follows from Proposition~\ref{th2.1}, Theorem~\ref{th2.3} and Lemma~\ref{th2.4} that $\Psi$ is a bijection with the desired properties.
\end{proof}

In order to construct the bijection $\Upsilon$ in Theorem~\ref{bij:sym}, we consider the set $\S_n(\pattern\,)$ of  all $(\pattern\,)$-avoiding permutations in $\S_n$, where $\pi\in\S_n$ avoids the generalized pattern $(\pattern\,)$ if there is no subsequence  $\pi_i\pi_{i+1}\pi_j$ of $\pi$ satisfying $\pi_i<\pi_{i+1}=\pi_j+1$. For $\pi=\pi_1\pi_2\cdots\pi_n\in\S_n$, let $\pi^c:=(n+1-\pi_1)(n+1-\pi_2)\cdots(n+1-\pi_n)$ be the {\em complement} of $\pi$. The following relation is obvious.

 \begin{lemma}\label{lem1}
 The transformation $\pi\mapsto(\pi^{-1})^{c}$ is a bijection between $\S_n(\sepattern\,)$ and $\S_n(\pattern\,)$ such that
 $$
 (\des, \iasc, \lmin, \lmax) \pi= (\iasc, \des, \lmax, \rmax)(\pi^{-1})^{c},
 $$
 where $\iasc(\pi)=n-1-\ides(\pi)$.
 \end{lemma}

The following is our second restricted double Eulerian equidistribution.
\begin{theorem}\label{mainth3}
There  is a bijection $\Phi: \S_n(\pattern\,)\rightarrow\mathcal{A}_n$  such  that  for any $\pi\in  \S_n(\pattern\,)$,
 $$
 (\DES, \IDES,   \LMAX, \RMAX)\pi=(\ASC, \DIST, \ZERO, \RMIN)\Phi(\pi).
 $$
 \end{theorem}

Before we give the proof of Theorem~\ref{mainth3}, we show how $\Upsilon$ can be constructed by combining bijections $\Psi$ and $\Phi$.
\begin{proof}[{\bf Proof of Theorem~\ref{bij:sym}}]
For each $s\in\A_n$, define $\Upsilon(s)=\Phi((\pi^{-1})^c)$, where $\pi=\Psi^{-1}(s)$. It then follows from Theorem~\ref{mainth}, Lemma~\ref{lem1} and Theorem~\ref{mainth3} that
\begin{align*}
(\asc,\rep,\max,\zero)s&= (\des,\iasc,\lmin,\lmax)\pi\\
&=(\iasc, \des, \lmax, \rmax)(\pi^{-1})^{c}\\
&=(\rep,\asc,\zero,\rmin)\Upsilon(s),
\end{align*}
as desired.
\end{proof}

We continue to prove Theorem~\ref{mainth3}.  Denote by $\mathcal{C}_n$ the set of sequences $(c_1,c_2,\ldots, c_n)\in\I_n$ with the  property that
\begin{itemize}
\item[(c)]  if $c_{i}\geq c_{i+1}$, then $i$ does not occur in  $(c_{i+1},c_{i+2},\ldots, c_{n})$.
\end{itemize}
For example, we have $\mathcal{C}_3=\{(0,0,0),(0,0,2),(0,1,0),(0,1,1),(0,1,2)\}$.
Next we aim to describe a map $\gamma$ from $\mathcal{C}_n$ to $\mathcal{A}_n$ in   terms of an algorithm.   Let $c=(c_1,c_2,\ldots, c_n)\in\mathcal{C}_n$ be the input sequence and suppose that $\NASC(c)=\{i_1, i_2, \ldots , i_k\}$ with $i_1>i_2>\cdots > i_k$. Do \\
\hspace*{0.5cm}  for $i=i_1,\ldots, i_k$:\\
\hspace*{1cm}  for $j=i+1, i+2, \ldots, n$:\\
 \hspace*{1.5cm}  if  $c_j>i$, then $c_j:=c_j-1$\\
and denote the resulting sequence by $\gamma(c)$.

The property (c) ensures the following fact of the algorithm $\gamma$.
\begin{fact}\label{gam:fact}
The relative order of  entries in $c\in\mathcal{C}_n$ and $\gamma(c)$ are the same.
\end{fact}

   This fact implies that  the positions of non-ascents in $c$ and $\gamma(c)$ coincide. So  the above procedure is easy to invert by the following {\em addition} algorithm. Let $\gamma(c)=(s_1,\ldots,s_n)$ be the input sequence. Do\\
\hspace*{0.5cm}  for $i=i_k, i_{k-1}\ldots, i_1$:\\
\hspace*{1cm}  for $j=i+1, i+2, \ldots, n$:\\
 \hspace*{1.5cm}  if  $s_j\geq i$, then $s_j:=s_j+1$.\\
Thus, the map $\gamma$ is injective.
By similar arguments as in the proof of Proposition~\ref{th2.1}, we can verify the following result.
\begin{proposition}\label{gam:th1}
The map $\gamma: \mathcal{C}_n\rightarrow\mathcal{A}_n$ is a bijection which preserves the quadruple  of set-valued statistics
$(\ASC,\DIST,\ZERO,\RMIN)$.
\end{proposition}
\begin{proof}
 First we shall prove that the map $\gamma$ is well defined, that is, for any $c\in\mathcal{C}_n$ we have $\gamma(c)=(s_1,s_2,\ldots, s_n)\in \A_n$. From  the definition of $\gamma$, we see that $s_i\leq i-1-\nasc(s_1s_2\ldots s_{i})=\asc(s_1s_2\ldots s_i)\leq \asc(s_1s_2\ldots s_{i-1})+1$. This implies that the resulting sequence $\gamma(c)$ is an ascent sequence.

 Recall that we have shown that the map $\gamma$  is injective. In order to show that the map $\gamma$ is a bijection, it remains to show that $\gamma$ is surjective.  For any $s\in \A_n$, we can  get a sequence $c=(c_1,c_2,\ldots, c_n)$ by applying the addition algorithm.
 We have to check that the resulting sequence has property (c), that is, if $s_i\geq s_{i+1}$ (equivalently $c_i\geq c_{i+1}$), then $i$ does not appear in $c_{i+1}c_{i+2}\cdots c_n$. This is true by the addition algorithm.
 Moreover, we have $c_i\leq s_i+\nasc(s_1,s_2\ldots, s_{i})\leq \asc(s_1,s_2,\ldots, s_{i})+\nasc(s_1,s_2,\ldots, s_{i})=i-1$.  This yields that $c\in \mathcal{C}_n$, and thus $\gamma$ is a bijection.

If follows from Fact~\ref{gam:fact} that the quadruple $(\ASC,\DIST,\ZERO,\RMIN)$ is preserved under $\gamma$, which completes the proof.
\end{proof}
It should be noted that $\gamma$ does not preserve the statistic $\MAX$.

\begin{lemma}\label{thchi3}
The permutation code $b$ induces a bijection between  $\S_n(\pattern\,)$ and $\mathcal{C}_n$. 
\end{lemma}
\begin{proof}
In view of Proposition~\ref{gam:th1} and Lemma~\ref{lem1}, the sets  $\S_n(\pattern\,)$ and $\mathcal{C}_n$ have the same cardinality. So it remains to show that  for any $\pi\in \S_n(\pattern\,)$, we have $b(\pi)=(c_1,c_2,\ldots, c_n)\in \mathcal{C}_n$. If $c_i\geq c_{i+1}$, then by Theorem~\ref{th2.3} we have $\pi_i<\pi_{i+1}$. Since $\pi$ avoids the pattern $(\pattern\,)$, the letter $\pi_{i+1}-1$ must appear in $\pi$ before $\pi_{i+1}$. This means that if $(I_v,\ell_v)$ is a labeled interval in the $i$-th slice of $\pi$ with $\pi_{i+1}\in I_v $, then  $\pi_{i+1}=min(I_v)$. By the construction of $b(\pi)$, the label of the interval containing  $0$ in the $i$-th slice, which is $i$, will disappear in  all slices after the $i$-th slice. This shows that $i$ will not appear in $(c_{i+1},c_{i+2},\ldots, c_{n})$, which completes the proof.
\end{proof}

 \begin{proof}[{\bf Proof of Theorem \ref{mainth3}}] Let $\Phi=\gamma \circ b$.  It follows from Proposition~\ref{gam:th1}, Theorem~\ref{th2.3} and Lemma~\ref{thchi3} that $\Phi$ is a bijection with the desired properties.
\end{proof}

\section{A generalization of Theorem~\ref{main:3}}\label{sec:4}

The purpose of this section is to prove a generalization of Theorem~\ref{main:3} with two new marginal statistics,  that we introduce below, corresponding to $\ealm$ on ascent sequences. For the sake of convenience, we index the maximals/zeros from left to right starting from $0$ (rather than from $1$). 

\begin{definition}
Let $s\in \T_n\setminus\{(0,1,\ldots,n-1)\}$.  Suppose that for $0\le i\le \max(s)-1$, the $i$-th maximal of $s$ is located at the $k_{i}$-th position, that is, $s_{k_{i}}=k_{i}-1$. Define
$$
\mpair(s)=\maxx\{0\le i\le \max(s)-1: s_{k_{i}+1}=s_{k_{i}}\}.
$$
For example, if $s=({\bf0},0,{\bf2},2,0,{\bf5},5,3)$, then  $\max(s)=3$, $k_0=1$, $k_1=3$, $k_2=6$ and consequently $\mpair(s)=2$. It is clear that $\mpair(s)$ exists for all $({\bf2-1})$-avoiding inversion sequences  $s\ne (0,1,2,\ldots,|s|-1)$.
\end{definition}

\begin{definition}
Let $s\in\A_n\setminus\{(0,0,\ldots,0)\}$. Suppose that for $0\le i\le \zero(s)-1$, the $i$-th zero of $s$ is located at the $k_{i}$-th position, that is, $s_{k_{i}}=0$. Define
\begin{align*}
\zpair(s)=\maxx\{0\le i\le \zero(s)-1:s_{k_{i}+1}=1\}.
\end{align*}
For example, if $s=(0,0,1,1,2,0,1,0)$, then $\zero(s)=4$ and $\zpair(s)=2$. It is clear that $\zpair(s)$ exists for all ascent sequences $s\ne (0,0,\ldots,0)$.
\end{definition}

The following is our generalization of Theorem~\ref{main:3}, involving  the marginal statistics $\zpair$, $\ealm$ and $\mpair$.
\begin{theorem}\label{T:main3}
For $n\geq1$, we have
\begin{equation}\label{E:rema2}
\sum_{s\in\T_n}
u^{\rep(s)}z^{\max(s)}w^{\mpair(s)}=\sum_{s\in\A_n}
u^{\rep(s)}z^{\max(s)}w^{\ealm(s)}=
\sum_{s\in\A_n}
u^{\asc(s)}z^{\zero(s)}w^{\zpair(s)},
\end{equation}
where we set $\mpair(0,1,\ldots,n-1)=\ealm(0,1,\ldots,n-1)=\zpair(0,0,\ldots,0)=0$.
\end{theorem}

\begin{remark}
We cannot establish the second equality of~\eqref{E:rema2}  via the approach in Section~\ref{sec:3}.
\end{remark}

The rest of this section is devoted to a bijective proof,  that makes inductive sense, of Theorem~\ref{T:main3}. 

\subsection{The  structure of $({\bf2-1})$-avoiding inversion sequences}

Before we prove the first equality of~\eqref{E:rema2}, we need to show some auxiliary lemmas.
\begin{lemma}\label{L:c1}
For any $0\le j\le p-1$, it holds that
\begin{align*}
&\,\quad \vert\{s\in \A_n\cap(\CMcal{S}_1\cup\CMcal{S}_2\cup\CMcal{S}_3): \ealm(s)=j,\rep(s)=k, \max(s)=p\}\vert\\
&=\vert\{s\in \A_{n-1}: \rep(s)=k-1, \max(s)=p\}\vert.
\end{align*}
\end{lemma}
\begin{proof}
For each $0\le j\le p-1$, let
\begin{align*}
A_{j}&=\{s\in \A_n\cap(\CS_1\cup\CS_2\cup\CS_3): \ealm(s)=j, \,\rep(s)=k,\max(s)=p\}.
\end{align*}
We aim to show that $\vert A_{i}\vert=\vert A_{i+1}\vert$ for any $0\le i<p-1$. To this end, we construct a one-to-one correspondence between the sets $A_i$ and $A_{i+1}$.

For any $s\in A_i$, we replace $i$ by $i+1$ on the $(p+1)$-th position. If $s_{p+2}=i+1$, then we further replace entries $d$ by $d-1$ whenever $d>p$. This yields an ascent sequence $s^*$ from $A_{i+1}$. It is evident that this transformation is reversible, that is, for any $s^*\in A_{i+1}$, we replace $i+1$ by $i$ on the $(p+1)$-th position. If $s^*_{p+2}=i+1$, then we further replace $d$ by $d+1$ whenever $d\ge p$.
In consequence, $\vert A_i\vert=\vert A_{i+1}\vert$ for all $0\le i<p-1$, hence $\vert A_i\vert=\vert A_{p-1}\vert$ for all $0\le i<p$. It remains to count $|A_{p-1}|$. Removing the $(p+1)$-th entries of all sequences from $A_{p-1}$ yields all ascent sequences $s\in\A_{n-1}$ such that $\rep(s)=k-1$ and $\max(s)=p$. Hence the proof is complete.
\end{proof}

 Let $\CMcal{F}$ be the set of $({\bf2-1})$-avoiding inversion sequences $s$  such that $\mpair(s)<\max(s)-1$ and the $(\mathsf{mpair}(s)+1)$-th maximal is either the last entry or followed immediately by a maximal. 
\begin{lemma}\label{L:d4}
There is a bijection $\psi: \CMcal{T}_{n}\cap \CMcal{F}\rightarrow\CMcal{T}_{n-1}$ such that $\max(\psi(s))=\mathsf{max}(s)-1$, $\rep(\psi(s))=\rep(s)$ and $\mpair(\psi(s))=\mpair(s)$.
\end{lemma}
\begin{proof}
For any $s\in \CMcal{T}_n\cap \CMcal{F}$ with $\mathsf{mpair}(s)=j$, if the $(j+1)$-th maximal is the last entry, then we simply remove it from $s$; otherwise, let $x$ be the $(j+2)$-th maximal of $s$, then we remove the $(j+2)$-th maximal $x$ from $s$ and for all $i\ge x$, we replace $i$ by $i-1$. In both cases, we are led to a  sequence $s^*=\psi(s)\in\T_{n-1}$ such that $\mathsf{rep}(s^*)=\mathsf{rep}(s)$, $\mathsf{max}(s^*)=\mathsf{max}(s)-1$ and $\mathsf{mpair}(s)=\mathsf{mpair}(s^*)$. It is easily seen that $\psi$ is a bijection.
\end{proof}
Next we introduce $\mathsf{mpos}(s)$ in order to divide the set of all $({\bf2-1})$-avoiding inversion sequences into two disjoint subsets.
\begin{definition}
Let $s$ be a $({\bf2-1})$-avoiding inversion sequence such that $\mathsf{max}(s)=p$ and $\mathsf{mpair}(s)=j$ ($0\le j<p$).  Suppose that for $0\le i\le p-1$, the $i$-th maximal of $s$ is located at the $k_{i}$-th position.  We call each position $\ell$, $\ell\ge k_j+2$ and $s_{\ell}=\ell-2$, a {\em critical maximal position} of $s$ and define
\begin{itemize}
\item $\mathsf{mpos}(s)=m+1$ if $m$ is the maximal integer such that the leftmost critical maximal position is greater than $k_{m}$;
\item $\mathsf{mpos}(s)=0$ if $s$ has no critical maximal positions.
\end{itemize}
For example, $\mathsf{mpos}(0,0,2,2,0,2,5)=2$ (as the sequence contains only one critical maximal   position at $7$) and $\mathsf{mpos}(0,0,2,2,0,2,4)=0$.
\end{definition}

Let $\T$ be the set of all $({\bf2-1})$-avoiding inversion sequences. We divide the set $\T':=\{s\in\T:|s|>\max(s)\}$ into the following disjoint subsets:
$$
\CMcal{J}_1:=\{s\in\CMcal{T}':  \mpos(s)=0\}\quad\text{and}\quad
\CMcal{J}_2:=\{s\in\CMcal{T}': \mpos(s)\ne0\}.
$$

\begin{lemma}\label{L:d1}
For any $0\le j\le p-1$, it holds that
\begin{align*}
&\,\quad \vert\{s\in \CMcal{T}_n\cap\CMcal{J}_{1}: \mpair(s)=j,\,\rep(s)=k,  \max(s)=p\}\vert\\
&=\vert\{s\in \CMcal{T}_{n-1}: \rep(s)=k-1,\max(s)=p\}\vert.
\end{align*}
\end{lemma}
\begin{proof}
For every $0\le j\le p-1$, let
\begin{align*}
C_j=\{s\in \CMcal{T}_n\cap\CMcal{J}_{1}:  \mpair(s)=j, \rep(s)=k, \max(s)=p\}.
\end{align*}
We aim  to show that $\vert C_i\vert=\vert C_{i+1}\vert$ for each $0\le i<p-1$. To this end, we construct a one-to-one correspondence between the set $C_i$ and the set $C_{i+1}$. For any $s\in C_i$, let $x_{\ell}$ be the $\ell$-th maximal of $s$ and let $y$ be the entry (if any) that is right after the $(i+1)$-th maximal $x_{i+1}$ of $s$. We distinguish two cases.  

If $s\in C_i\cap(\T\setminus\CMcal{F})$, then  by definition $y$ must exist in $s$ and $y\leq x_{i+1}-2$. The sequence $s^*$ is constructed from $s$ as follows: remove the integer that is right after the $i$-th maximal $x_i$ of $s$, replace $y$ by $x_{i+1}$ and insert $y$ right before the $(i+1)$-th maximal $x_{i+1}$; see Fig.~\ref{F:b1}.

\begin{figure}[ht]
\centering
\includegraphics[scale=1.0]{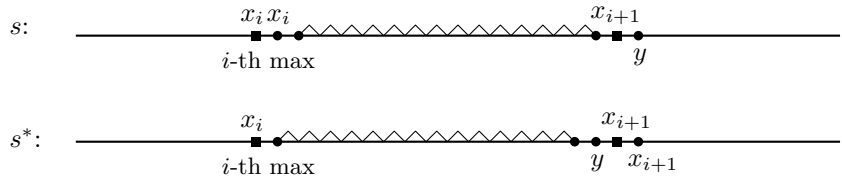}
\caption{The bijection between the sets $C_i$ and $C_{i+1}$ if $y\le x_{i+1}-2$, where all maximal elements are marked with squares. \label{F:b1}}
\end{figure}

If $s\in C_i\cap\CMcal{F}$, then by definition, either $y=x_{i+2}=x_{i+1}+1$ or $x_{i+1}$ is the last entry of $s$. The construction of $s^*$ from $s$ is given as follows: remove the $(i+1)$-th maximal $x_{i+1}$, replace (entries) $d$ by $d'=d+1$ if $x_{i}\le d\le x_{i+1}-1$, insert $x_i$ right before the leftmost $x_i'=x_{i}+1$; see Fig.~\ref{F:b2}.

\begin{figure}[ht]
\centering
\includegraphics[scale=1.0]{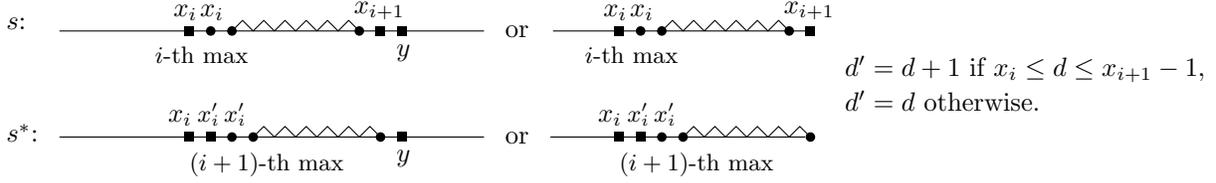}
\caption{The bijection between the sets $C_i$ and $C_{i+1}$ when $y=x_{i+2}$ or $x_{i+1}$ is the last entry of $s$, where maximal elements are marked with squares. \label{F:b2}}
\end{figure}

In both cases the construction $s\mapsto s^*$ leads to (notice that $s\in\CMcal{J}_1$)
\[
\begin{array}{rllll}
&\mathsf{rep}(s)=\mathsf{rep}(s^*), &\mathsf{mpos}(s)=\mathsf{mpos}(s^*)=0,\\
&\mathsf{max}(s)=\mathsf{max}(s^*), & \mathsf{mpair}(s)=\mathsf{mpair}(s^*)-1.
\end{array}
\]
In consequence, $s^*\in C_{i+1}$. The map $s\mapsto s^*$ is reversible, because for any $s^*\in C_{i+1}$, if the $(i+1)$-th maximal follows immediately after the $i$-th maximal, then $s^*$ must come from the construction shown in Fig.~\ref{F:b2}; otherwise $s^*$ must come from the construction shown in Fig.~\ref{F:b1}. It follows that $s\mapsto s^*$ is a bijection and $|C_i|=|C_{i+1}|$ for each $0\le i<p-1$. Hence $\vert C_i\vert=\vert C_{p-1}\vert$ for all $0\le i<p$. It remains to count the number $|C_{p-1}|$. Removing the $(p-1)$-th maximals of all sequences in $C_{p-1}$ yields all sequences $s\in\CMcal{T}_{n-1}$ such that $\mathsf{rep}(s)=k-1$ and $\mathsf{max}(s)=p$, which completes the proof.
\end{proof}
The following is our main bijection for $\CMcal{T}_n$, which consists  of a series of fundamental transformations. The parameter $\mpos$ plays an important role in our construction. 
\begin{lemma}\label{L:d2}
 Let $\CMcal{F}^c=\T\setminus\CMcal{F}$. There is a bijection
\begin{align*}
\vartheta: \{(s,i):s\in \CMcal{T}_n\cap\CMcal{F}^c, i<\mpair(s)\}\rightarrow \{s\in \CMcal{T}_n\cap\CMcal{J}_{2}: \mpair(s)=i\}
\end{align*}
such that if $\vartheta(s,i)=\hat{s}$, then $\rep(\hat{s})=\rep(s)$ and $\max(\hat{s})=\max(s)-1$.
\end{lemma}
\begin{proof}
Let $s\in\CMcal{T}_n\cap\CMcal{F}^c$ be a sequence with $\mathsf{mpair}(s)=j$ and $\mathsf{max}(s)=p+1$. Let $x_{\ell}$ be the $\ell$-th maximal of $s$, then either the $(j+1)$-th maximal $x_{j+1}$ is followed by an integer $y$ such that $y\le x_{j+1}-2$, or $x_{j}$ is the last maximal, i.e., $j=p$. For every pair $(s,i)$ where $0\le i<j\le p$, we start with constructing a new sequence $\hat{s}^{(1)}$ from the pair $(s,j-1)$. The sequence $\hat{s}^{(1)}$ is constructed from $s$ by the following step $\CMcal{M}_0$ (see Fig.~\ref{F:21}):
\begin{itemize}
\item remove the $j$-th maximal $x_j$;
\item insert $x_{j-1}$ right after the $(j-1)$-th maximal $x_{j-1}$.
\end{itemize}
To be precise, let $k_{\ell}$ be the position of the $\ell$-th ($0\leq \ell\leq p$) maximal of $s$. Then by the above construction, the leftmost critical maximal position of $\hat{s}^{(1)}$ is $k_j+1$. Hence $\hat{s}^{(1)}\in \CMcal{T}_n\cap\CMcal{J}_2$ with
\[
\begin{array}{rlllll}
&\mathsf{mpair}(\hat{s}^{(1)})=j-1=\mathsf{mpair}(s)-1, & \mathsf{max}(\hat{s}^{(1)})=p=\mathsf{max}(s)-1,\\
&\mathsf{mpos}(\hat{s}^{(1)})=j=\mathsf{mpair}(s), &\mathsf{rep}(\hat{s}^{(1)})=\mathsf{rep}(s).
\end{array}
\]
\begin{figure}[htbp]
\centering
\includegraphics[scale=1.0]{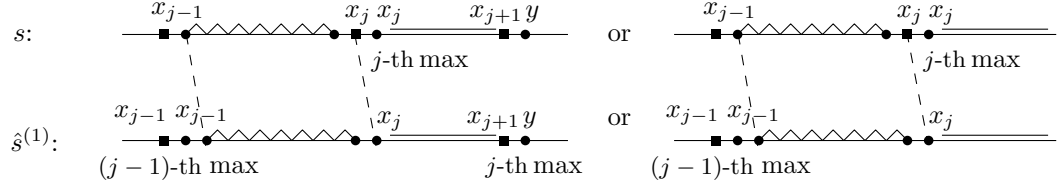}
\caption{The bijection from the pair $(s,j-1)$ to $\hat{s}^{(1)}$ where $y\le x_{j+1}-2$ or $x_{j}$ is the last maximal of $s$. The maximal elements are marked with  squares.\label{F:21}}
\end{figure}

If $i=j-1$, we stop and set $\hat{s}=\vartheta(s,j-1)=\hat{s}^{(1)}$; otherwise $i<j-1$, and we continue with the pair $(\hat{s}^{(1)},j-2)$. We need to distinguish two cases.

If $k_{j-2}+1=k_{j-1}$, then the sequence $\hat{s}^{(2)}$ is constructed from $\hat{s}^{(1)}$ by the following step $\CMcal{M}_1$ (see Fig.~\ref{F:22}). If $x_{j-1}$ is not the last maximal of $\hat{s}^{(1)}$, let $x$ be the $j$-th maximal of $\hat{s}^{(1)}$, that is, $x=x_{j+1}$ in Fig.~\ref{F:21}, then
\begin{itemize}
\item remove the $(j-1)$-th maximal $x_{j-1}$;
\item replace $d$ by $d'=d-1$ if $x_{j-1}\le d<x$;
\item insert $(x-1)$ right before the leftmost $x$;
\end{itemize}
otherwise $x_{j-1}$ is the last maximal, then
\begin{itemize}
\item remove the $(j-1)$-th maximal $x_{j-1}$;
\item replace $d$ by $d'=d-1$ if $d\ge x_{j-1}$;
\item insert a new maximal as the last entry.
\end{itemize}
\begin{figure}[ht]
\centering
\includegraphics[scale=1.0]{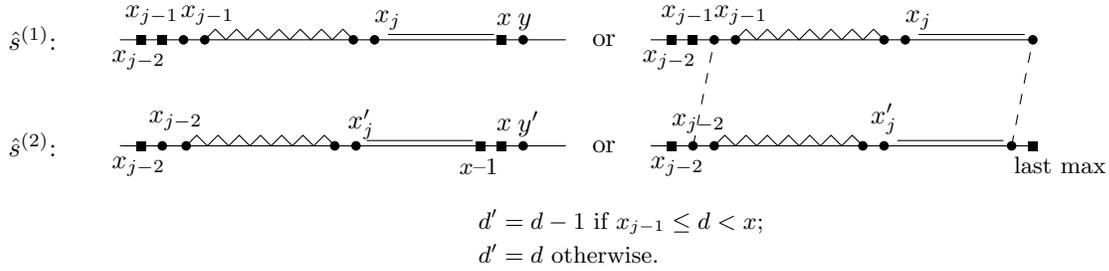}
\caption{The map $\hat{s}^{(1)}\mapsto\hat{s}^{(2)}$ when $k_{j-2}+1=k_{j-1}$ and the maximal elements are marked with  squares.\label{F:22}}
\end{figure}
In both cases $x_{j}\ge x_{j-1}+1$, so $x_j'=x_j-1$ and it can be easily seen that $k_j$ is the leftmost critical maximal position of $\hat{s}^{(2)}$. Furthermore, one can check that $\hat{s}^{(2)}\in\CMcal{T}_n\cap\CMcal{J}_{2}$ with
\[
\begin{array}{rlllll}
&\mathsf{mpair}(\hat{s}^{(2)})=j-2=\mathsf{mpair}(s)-2, &\mathsf{max}(\hat{s}^{(2)})=\mathsf{max}(s)-1,\\
&\mathsf{mpos}(\hat{s}^{(2)})=j-1=\mathsf{mpair}(s)-1, &\mathsf{rep}(\hat{s}^{(2)})=\mathsf{rep}(s).
\end{array}
\]


If $k_{j-2}+1<k_{j-1}$, then the sequence $\hat{s}^{(2)}$ is constructed from $\hat{s}^{(1)}$ by the following step $\CMcal{M}_2$ (see Fig.~\ref{F:23}):
\begin{itemize}
\item switch the integers on the $(k_{j-1}-1)$-th and the $(k_{j-1}+1)$-th positions;
\item remove the leftmost $x_{j-1}$;
\item insert $x_{j-2}$ right after the leftmost $x_{j-2}$.
\end{itemize}
\begin{figure}[ht]
\centering
\includegraphics[scale=1.0]{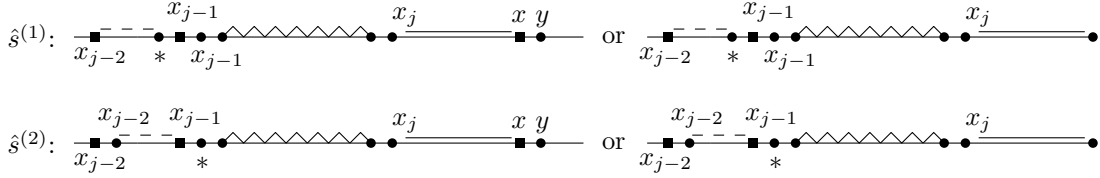}
\caption{The map $\hat{s}^{(1)}\mapsto\hat{s}^{(2)}$ when $k_{j-2}+1<k_{j-1}$, where $\ast$ denotes an integer less than $(x_{j-1}-1)$ and the maximal elements are marked with squares.\label{F:23}}
\end{figure}
By construction the leftmost critical maximal position of $\hat{s}^{(2)}$ is still $(k_j+1)$. One can easily check that $\hat{s}^{(2)}\in \CMcal{T}_n\cap\CMcal{J}_2$ with
\[
\begin{array}{rlllll}
&\mathsf{mpair}(\hat{s}^{(2)})=j-2=\mathsf{mpair}(s)-2, & \mathsf{max}(\hat{s}^{(2)})=\mathsf{max}(s)-1,\\
&\mathsf{mpos}(\hat{s}^{(2)})=j=\mathsf{mpair}(s), &\mathsf{rep}(\hat{s}^{(2)})=\mathsf{rep}(s).
\end{array}
\]

If $i=j-2$, we stop and set $\hat{s}=\vartheta(s,j-2)=\hat{s}^{(2)}$; otherwise $i<j-2$ and we continue with pair $(\hat{s}^{(2)},j-3)$. If $k_{j-3}+1=k_{j-2}$, we repeat step $\CMcal{M}_1$; otherwise we repeat step $\CMcal{M}_2$. We continue this process until we reach $i$ and we set $\hat{s}=\vartheta(s,i)=\hat{s}^{(j-i)}$ where $|\hat{s}|=|s|$, $\mpair(\hat{s})=i$, $\max(\hat{s})=\max(s)-1$ and $\rep(\hat{s})=\rep(s)$.

The construction $\vartheta$ is reversible. For any  sequence $\hat{s}\in\CMcal{T}_n\cap\CMcal{J}_2$ with $\mpair(\hat{s})=i$, if $\hat{s}\in \CMcal{F}$, then $\hat{s}$ is produced from step $\CMcal{M}_1$; otherwise if $\hat{s}\in \CMcal{F}^c$ and $\mathrm{mpos}(\hat{s})=i+1$, then $\hat{s}$ is produced from step $\CMcal{M}_0$. If $\hat{s}\in \CMcal{F}^c$ and $\mathrm{mpos}(\hat{s})\ne i+1$, then $\hat{s}$ is produced from step $\CMcal{M}_2$. This implies that $\vartheta$ is recursively reversible, and therefore $\vartheta$ is a bijection. Hence the proof is complete.
\end{proof}
\begin{example}

For $s=(0,0,0,3,4,5,5,7,1,5)\in\CMcal{T}_{10}\cap\CMcal{F}^c$ with $\mathsf{mpair}(s)=3$ and $i=0$, under the bijection $\vartheta$ we obtain that $\hat{s}^{(1)}=\vartheta(s,2)=(0,0,0,3,4,4,5,7,1,5)$,
$\hat{s}^{(2)}=\vartheta(s,1)=(0,0,0,3,3,4,6,7,1,4)$, and $\hat{s}^{(3)}=\vartheta(s,0)=(0,0,0,3,0,4,6,7,1,4)$.
\end{example}

Now we are ready to prove the first equality of~\eqref{E:rema2}.

\begin{proof}[{\bf Proof of the first equality of~\eqref{E:rema2}}]
 We will prove the first equality of~\eqref{E:rema2} together with
\begin{align}\label{E:21}
\sum_{s\in\A_n\cap(\CMcal{S}_1\cup\CMcal{S}_2\cup\CMcal{S}_3)}
x^{\rep(s)}q^{\max(s)}w^{\ealm(s)}
=\sum_{s\in\CMcal{T}_n\cap\CMcal{J}_1}
x^{\rep(s)}q^{\max(s)}w^{\mpair(s)},
\end{align}
\begin{align}
\label{E:22}\sum_{s\in\A_n\cap\CMcal{S}_4}
x^{\rep(s)}q^{\max(s)}w^{\ealm(s)}
=\sum_{s\in\CMcal{T}_n\cap\CMcal{J}_2}
x^{\rep(s)}q^{\max(s)}w^{\mpair(s)}
\end{align}
by induction on the number $|s|-\max(s)$ for all sequences $s$. For $|s|=\max(s)$, the first equality of~\eqref{E:rema2} is trivial.

Suppose that the triple $(\rep,\max,\ealm)$ on ascent sequences $s$ with $|s|-\max(s)=N-1$ is equidistributed to the triple $(\rep,\max,\mpair)$ on $({\bf2-1})$-avoiding inversion sequences $s$ with $|s|-\max(s)=N-1$. It suffices to prove that
\begin{align*}
&\quad\vert\{s\in\A_n: \ealm(s)=j, \rep(s)=k, \max(s)=p\}\vert\\
&=\vert\{s\in\CMcal{T}_n:\mpair(s)=j, \rep(s)=k, \max(s)=p\}\vert,
\end{align*}
whenever $n-p=N$. In combination of Lemma~\ref{L:c1} and Lemma~\ref{L:d1}, it follows immediately that for all $0\le j<p$,
\begin{align*}
\,\quad &\vert\{s\in \A_n\cap(\CMcal{S}_1\cup\CMcal{S}_2\cup\CMcal{S}_3): \ealm(s)=j,\rep(s)=k, \max(s)=p\}\vert\\
=& \vert\{s\in \CMcal{T}_n\cap\CMcal{J}_1: \mpair(s)=j, \rep(s)=k, \max(s)=p\}\vert.
\end{align*}
In other words, (\ref{E:21}) holds. It remains to show that
\begin{align}\label{E:t3s41}
\,\quad &\vert\{s\in \A_n\cap\CMcal{S}_4:
 \ealm(s)=j,\rep(s)=k, \max(s)=p\}\vert\\
\nonumber=& \vert\{s\in \CMcal{T}_n\cap\CMcal{J}_2:
\mpair(s)=j,\,\rep(s)=k, \,\max(s)=p\}\vert.
\end{align}
By induction hypothesis, Lemma~\ref{L:case3} and Lemma~\ref{L:d4}, we know that
\begin{align*}
\,\quad &\vert\{s\in \A_n\cap \CMcal{P}:
\ealm(s)=j, \rep(s)=k, \max(s)=p+1\}\vert\\
\nonumber=& \vert\{s\in \A_{n-1}:
\,\ealm(s)=j,\,\rep(s)=k, \,\max(s)=p\}\vert\\
\nonumber=& \vert\{s\in \CMcal{T}_{n-1}:
\mpair(s)=j,\rep(s)=k, \max(s)=p\}\vert\\
\nonumber=&\vert\{s\in \CMcal{T}_n\cap \CMcal{F}:
\,\mpair(s)=j,\,\rep(s)=k, \,\max(s)=p+1\}\vert,
\end{align*}
which together with (again by induction hypothesis)
\begin{align*}
\,\quad &\vert\{s\in \A_n:
\,\ealm(s)=j,\,\rep(s)=k, \,\max(s)=p+1\}\vert\\
\nonumber=&\vert\{s\in \CMcal{T}_n:
\,\mpair(s)=j,\,\rep(s)=k, \,\max(s)=p+1\}\vert,
\end{align*}
leads to
\begin{align*}
\,\quad &\vert\{s\in \A_n\cap \CMcal{P}^c:
\,\ealm(s)=j,\,\rep(s)=k, \,\max(s)=p+1\}\vert\\
\nonumber=&\vert\{s\in \CMcal{T}_n\cap \CMcal{F}^c:
\,\mpair(s)=j,\,\rep(s)=k, \,\max(s)=p+1\}\vert.
\end{align*}
In view of Lemmas~\ref{L:c3} and~\ref{L:d2}, (\ref{E:t3s41}) is true. In other words, (\ref{E:22}) is true and the proof is complete.
\end{proof}

\subsection{Further structure of ascent sequences}
We will prove the second equality of~\eqref{E:rema2} via a decomposition  regarding the pair $(\asc,\zero)$ on ascent sequences.
 Let $\CMcal{G}$ be the set of sequences $s\in\A$ such that $\zpair(s)<\zero(s)-1$ and the $(\mathsf{zpair}(s)+1)$-th zero is either the last entry or followed immediately by a zero.
\begin{lemma}\label{L:c4}
There is a bijection $\varphi: \A_{n}\cap \CMcal{G}\rightarrow\A_{n-1}$ such that $\zero(\varphi(s))=\zero(s)-1$, $\asc(\varphi(s))=\asc(s)$ and $\zpair(\varphi(s))=\zpair(s)$.
\end{lemma}
\begin{proof}
For any $s\in \A_{n}\cap \CMcal{G}$ with $\mathsf{zpair}(s)=j$, by definition, either the $(j+1)$-th zero is the last entry of $s$ or it is followed immediately by a zero. In both cases, we remove the $(j+1)$-th zero of $s$, which leads to an ascent sequence $s^*=\varphi(s)$ such that $\asc(s^*)=\asc(s)$, $\zero(s^*)=\zero(s)-1$ and $\mathsf{zpair}(s)=\mathsf{zpair}(s^*)$. It is easily seen that $\varphi$ is a bijection.
\end{proof}
\begin{definition}
Let $s$ be an ascent sequence such that $\zero(s)=p$ and $\mathsf{zpair}(s)=j$ ($0\le j< p$). Suppose that for $0\le i\le p-1$, the $i$-th zero of $s$ is located at the $k_{i}$-th position. We call every $\ell$,  $s_{\ell}=1$ and $\ell\ge k_j+2$, a {\em critical zero position} of $s$ and define
\begin{itemize}
\item $\mathsf{zpos}(s)=m+1$ if $m$ is the maximal integer such that the leftmost critical zero position is greater than $k_{m}$;
\item $\mathsf{zpos}(s)=0$ if $s$ has no critical zero positions.
\end{itemize}
For example, $\zpos(0,1,2,0,1,3,2,1,0)=2$ (as the sequence  contains only one critical zero  position at $8$) and $\zpos(0,1,2,0,1,3,2,0)=0$.
\end{definition}


We divide the set  $\tilde{\A}:=\{s\in\A: |s|>\zero(s)\}$ into the following disjoint subsets:
$$
\CMcal{R}_1:=\{s\in\tilde{\A}:  \mathsf{zpos}(s)=0\}\quad\text{and}\quad\CMcal{R}_2:=\{s\in\tilde{\A}:  \mathsf{zpos}(s)\ne0\}.
$$
\begin{lemma}\label{L:c2}
For any $0\le j\le p-1$, it holds that
\begin{align*}
&\,\quad \vert\{s\in \A_n\cap\CMcal{R}_1: \mathsf{zpair}(s)=j, \asc(s)=k,  \zero(s)=p\}\vert\\
&=\vert\{s\in \A_{n-1}: \asc(s)=k-1, \zero(s)=p\}\vert.
\end{align*}
\end{lemma}
\begin{proof}
For every $0\le j\le p-1$, let
\begin{align*}
B_j=\{s: s\in \A_n\cap\CMcal{R}_1,\, \mathsf{zpair}(s)=j,\,\asc(s)=k,\,\zero(s)=p\},
\end{align*}
then we want to prove that $\vert B_i\vert=\vert B_{i+1}\vert$ for any $0\le i<p-1$. To this end, we construct a one-to-one correspondence between the sets $B_i$ and $B_{i+1}$. For any $s\in B_i$, let $y$  be the entry (if any) that is right after the $(i+1)$-th zero.

The sequence $s^*$ is constructed from $s$ as follows: remove the integer $1$ right after the $i$-th zero, replace  $d$ by $d-1$ for all $d$ between $i$-th zero and $(i+1)$-th zero, and insert $1$ right after the $(i+1)$-th zero; see Fig.~\ref{F:c1}. It follows that
\[
\begin{array}{rlllll}
&\mathsf{zpos}(s)=\mathsf{zpos}(s^*)=0, &\mathsf{asc}(s)=\mathsf{asc}(s^*),\\
&\mathsf{zero}(s)=\mathsf{zero}(s^*), & \mathsf{zpair}(s)=\mathsf{zpair}(s^*)-1.
\end{array}
\]

\begin{figure}[ht]
\centering
\includegraphics[scale=1.0]{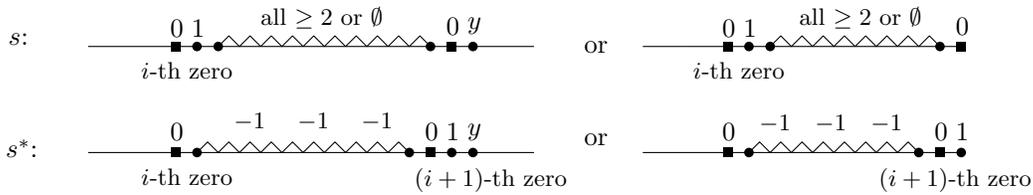}
\caption{The bijection between the sets $B_i$ and $B_{i+1}$ if $y
\neq 1$ exists in $s$ or the $(i+1)$-th zero is the last entry of $s$, where all zeros are marked with  squares. \label{F:c1}}
\end{figure}
In consequence, $s^*\in B_{i+1}$. The map $s\mapsto s^*$ is reversible, because for $s^*\in B_{i+1}$, by definition the $(i+1)$-th zero is followed immediately by an integer $1$ and if this $1$ is the last entry, then $s^*$ must come from the construction shown on the right of Fig.~\ref{F:c1}; otherwise $s^*$ comes from the construction shown on the left of Fig.~\ref{F:c1}. This implies that $s\mapsto s^*$ is a bijection and $|B_i|=|B_{i+1}|$ for all $0\le i<p-1$. Hence $|B_i|=|B_{0}|$ for all $0\le i<p$. It remains to count $|B_0|$. Removing the unique entry $1$ and replacing nonzero $d$ by $d-1$ for all sequences in $B_0$, yields all sequences $s\in\A_{n-1}$ such that $\asc(s)=k-1$ and $\zero(s)=p$, completing the proof.
\end{proof}
\begin{lemma}\label{L:T3}
Let  $\CMcal{G}^c=\A\setminus\CMcal{G}$. There is a bijection
\begin{align*}
\theta: \{(s,i):s\in \A_n\cap \CMcal{G}^c, i<\mathsf{zpair}(s)\}\rightarrow \{s\in \A_n\cap\CMcal{R}_2: \mathsf{zpair}(s)=i\}
\end{align*}
such that if $\hat{s}=\theta(s,i)$, then $\asc(\hat{s})=\asc(s)$ and $\zero(\hat{s})=\zero(s)-1$ .
\end{lemma}
\begin{proof}
Let $s\in\A_n\cap\CMcal{G}^c$ be a sequence with $\mathsf{zpair}(s)=j$ and $\mathsf{zero}(s)=p+1$, then either the $(j+1)$-th zero is followed immediately by an integer $y$ such that $y\ge 2$ or the $j$-th zero is the last zero. For every pair $(s,i)$ where $0\le i<j\le p$, we start with constructing a new sequence $\hat{s}^{(1)}$ from the pair $(s,j-1)$. We assume that  the $\ell$-th  ($0\leq\ell\leq p$) zero of $s$ is located at the $k_{\ell}$-th position. The sequence $\hat{s}^{(1)}$ is constructed by the following step $\CMcal{Z}_0$ (see Fig.~\ref{F:1}):
\begin{itemize}
\item replace $s_i$ by $s_i+1$ for all $k_{j-1}<i<k_j$;
\item remove the zero on the $k_j$-th position;
\item add the integer $1$ right after the $(j-1)$-th zero.
\end{itemize}
Clearly the leftmost critical zero position of $\hat{s}^{(1)}$ is $(k_j+1)$. Hence $\hat{s}^{(1)}\in \A_n\cap\CMcal{R}_2$ with
\[
\begin{array}{rlllll}
&\mathsf{zpair}(\hat{s}^{(1)})=j-1=\mathsf{zpair}(s)-1, \,\,& \mathsf{zero}(\hat{s}^{(1)})=\mathsf{zero}(s)-1,\\
&\mathsf{zpos}(\hat{s}^{(1)})=j=\mathsf{zpair}(s),&
\mathsf{asc}(\hat{s}^{(1)})=\mathsf{asc}(s).
\end{array}
\]
\begin{figure}[htbp]
\centering
\includegraphics[scale=1.0]{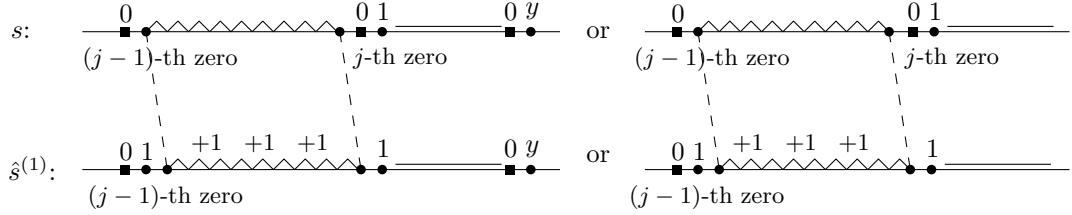}
\caption{The bijection from the pair $(s,j-1)$ to $\hat{s}^{(1)}$ if $y\ge 2$ or the $j$-th zero is the last zero, where all zeros are marked with  squares.\label{F:1}}
\end{figure}

If $i=j-1$, we stop and set $\hat{s}=\theta(s,j-1)=\hat{s}^{(1)}$; otherwise $i<j-1$ and we continue with the pair $(\hat{s}^{(1)},j-2)$. We distinguish two cases.

If $k_{j-2}+1=k_{j-1}$, then $\hat{s}^{(2)}$ is constructed by the following step $\CMcal{Z}_1$ (see Fig.~\ref{F:2}):
\begin{itemize}
\item remove the $(j-1)$-th zero;
\item insert a zero right before $y$ if $(j-1)$-th zero is not the last zero; insert a zero as the last entry, otherwise.
\end{itemize}
It is clear that for both cases $\hat{s}^{(2)}\in\A_n\cap\CMcal{R}_2$ with
\[
\begin{array}{rlllll}
&\mathsf{zpair}(\hat{s}^{(2)})=j-2=\mathsf{zpair}(s)-2, & \mathsf{zero}(\hat{s}^{(2)})=\mathsf{zero}(s)-1,\\
&\mathsf{zpos}(\hat{s}^{(2)})=j-1=\mathsf{zpair}(s)-1,&
\mathsf{asc}(\hat{s}^{(2)})=\mathsf{asc}(s).
\end{array}
\]
\begin{figure}[ht]
\centering
\includegraphics[scale=1.0]{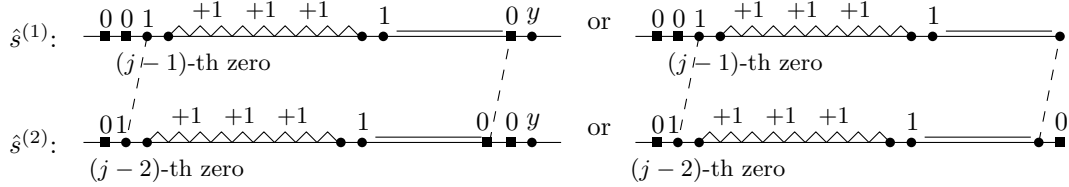}
\caption{The map $\hat{s}^{(1)}\mapsto\hat{s}^{(2)}$ if $k_{j-2}+1=k_{j-1}$, where $y\ge 2$ or the $(j-1)$-th zero is the last zero, where all zeros are marked with squares.\label{F:2}}
\end{figure}

If $k_{j-2}+1<k_{j-1}$, then $\hat{s}^{(2)}$ is constructed by the following step $\CMcal{Z}_2$ (see Fig.~\ref{F:3}):
\begin{itemize}
\item replace $s_i$ by $s_i+1$ for all $k_{j-2}<i<k_{j-1}$;
\item move the integer $1$ that is on the $(k_{j-1}+1)$-th position to the $(k_{j-2}+1)$-th position if an integer more than $1$ is located at the $(k_{j-1}+2)$-th position; move two integers $0,1$ that are on the $k_{j-1}$-th and the $(k_{j-1}+1)$-th positions respectively to the $k_{j-2}$-th and the $(k_{j-2}+1)$-th positions, otherwise.
\end{itemize}
It is clear that $\hat{s}^{(2)}$ is an ascent sequence in $\A_n\cap\CMcal{R}_2$ with
\[
\begin{array}{rlllll}
&\mathsf{zpair}(\hat{s}^{(2)})=j-2=\mathsf{zpair}(s)-2, \,\,& \mathsf{zero}(\hat{s}^{(2)})=\mathsf{zero}(s)-1,\\
&\mathsf{zpos}(\hat{s}^{(2)})=j=\mathsf{zpair}(s),&
\mathsf{asc}(\hat{s}^{(2)})=\mathsf{asc}(s).
\end{array}
\]
\begin{figure}[ht]
\centering
\includegraphics[scale=1.0]{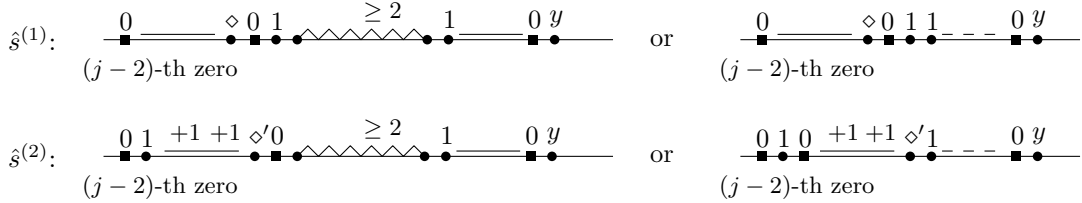}
\caption{The map $\hat{s}^{(1)}\mapsto\hat{s}^{(2)}$ if $k_{j-2}+1<k_{j-1}$, where $y\ge 2$, $\diamond'=\diamond+1\ge 2$ and all zeros are marked with squares. We omit the picture if the $(j-1)$-th zero is the last zero.\label{F:3}}
\end{figure}

If $i=j-2$, we stop and set $\hat{s}=\theta(s,j-2)=\hat{s}^{(2)}$; if $i<j-2$, we continue with pair $(\hat{s}^{(2)},j-3)$. If $k_{j-3}+1=k_{j-2}$, we repeat step $\CMcal{Z}_1$; otherwise we repeat step $\CMcal{Z}_2$. We continue this process until we reach $i$ and we set $\hat{s}=\theta(s,i)=\hat{s}^{(j-i)}$ where $|\hat{s}|=|s|$, $\mathsf{zpair}(\hat{s})=i$, $\zero(\hat{s})=\zero(s)-1$ and $\asc(\hat{s})=\asc(s)$.

The construction $\theta$ is reversible. For any sequence $\hat{s}\in\A_n\cap\CMcal{R}_2$ with $\mathsf{zpair}(\hat{s})=i$, if $\hat{s}\in\CMcal{G}$, then $\hat{s}$ is produced from step $\CMcal{Z}_1$; otherwise if $\hat{s}\in \CMcal{G}^c$ and $\mathsf{zpos}(\hat{s})=i+1$, then $\hat{s}$ is produced from step $\CMcal{Z}_0$. If $\hat{s}\in \CMcal{G}^c$ and $\mathsf{zpos}(\hat{s})\ne i+1$, then $\hat{s}$ is produced from step $\CMcal{Z}_2$.  More precisely, if the $i$-th zero of $\hat{s}$ is followed by $1\,0$, then we use the construction shown on the right of Fig.~\ref{F:3} to reverse $\theta$; if the $i$-th zero of $\hat{s}$ is followed by $1\,x$, with $x\ge 2$, then we use the construction shown on the left of Fig.~\ref{F:3} to reverse $\theta$. This implies that $\theta$ is recursively reversible, therefore $\theta$ is a bijection. Hence the proof is complete.
\end{proof}
\begin{example}
For $s=(0,1,2,0,0,1,2,4,1,2,0,2)$ with $\mathsf{zpair}(s)=2$ and $i=0$, under the bijection $\theta$ we obtain $\hat{s}^{(1)}=\theta(s,1)=(0,1,2,0,1,1,2,4,1,2,0,2)$ and
$\hat{s}^{(2)}=\theta(s,0)=(0,1,0,2,3,1,2,4,1,2,0,2)$.
\end{example}


\begin{proof}[{\bf Proof of the second equality of~\eqref{E:rema2}}]
Using Lemmas~\ref{L:c4},~\ref{L:c2} and~\ref{L:T3}, we
 can prove 
\begin{align*}
\sum_{s\in\A_n\cap(\CMcal{S}_1\cup\CMcal{S}_2\cup\CMcal{S}_3)}
x^{\rep(s)}q^{\max(s)}w^{\ealm(s)}
&=\sum_{s\in\A_n\cap\CMcal{R}_1}
x^{\asc(s)}q^{\zero(s)}w^{\mathsf{zpair}(s)},\\
\sum_{s\in\A_n\cap\CMcal{S}_4}
x^{\rep(s)}q^{\max(s)}w^{\ealm(s)}
&=\sum_{s\in\A_n\cap\CMcal{R}_2}
x^{\asc(s)}q^{\zero(s)}w^{\mathsf{zpair}(s)}
\end{align*}
by induction on the number $|s|-\max(s)$ for all ascent sequences $s$.  We omit the details of the discussions, since they are similar to those in the proof of the first equality of~\eqref{E:rema2}.
\end{proof}

\section*{Recent developments} 

 A proof of Conjecture~\ref{conj1ref}  using Theorem~\ref{T:gen}  and the machinery
of basic hypergeometric series was recently found and will be featured in a separate paper~\cite{js}. It follows from Corollary~\ref{zero:max} and Theorem~\ref{bij:sym} that the pair $(\rmin,\zero)$ is symmetric on $\A_n$, which has an alternative proof provided by Chen, Yan and Zhou~\cite{cyz}.

\section*{Acknowledgements}

We thank the referees for carefully reading the paper and providing insightful comments and suggestions.

Fu was supported by the National Science Foundation of China  grant~11501061 and the Fundamental Research Funds for the Central Universities No.~2018CDXYST0024.

Jin gratefully acknowledges supports from the German Research Foundation DFG, JI 207/1-1, the Austrian Research Fund FWF, project SFB F50-02/03 and grant P 32305, and FWF-MOST (Austria-Taiwan) project I~2309-N35.

Lin was supported by the National Science Foundation of China grants 11871247 and 11501244,   the project of Qilu Young Scholars of Shandong University and the Austrian Research Fund FWF, START grant Y463 and SFB grant F50-10. 

Yan was supported by the National Science Foundation of China  grant 11671366. 

Zhou was supported by the National Science Foundation of China  grants 11801378 and 11626158, and the Zhejiang Provincial Natural Science Foundation of China No.~LQ17A010004.


\begin{thebibliography}{99}

\bibitem{bv}
J.L. Baril, V. Vajnovszki,
A permutation code preserving a double Eulerian bistatistic,
Discrete Appl. Math., {\bf224} (2017), 9--15.


\bibitem{bcdk}M. Bousquet-M\'elou, A. Claesson, M. Dukes and S. Kitaev,
$({\bf2+2})$-free posets, ascent sequences and pattern avoiding permutations,
J. Combin. Theory Ser. A, {\bf117} (2010), 884--909.

\bibitem{cyz} D. Chen, S.H.F. Yan and R.D.P. Zhou, Equidistributed Statistics on Fishburn Matrices and Permutations, Electron. J. Combin.,  {\bf26 (1)} (2019), \#P1.11.

\bibitem{cl} A. Claesson and S. Linusson, n! matchings, n! posets, Proc. Amer. Math. Soc., {\bf139} (2011), 435--449.

\bibitem{dp}M. Dukes and R. Parviainen, Ascent sequences and upper triangular matrices containing non-negative integers, Electronic J. Combin., {\bf17} (2010), \#R53.

\bibitem{dkrs}M. Dukes, S. Kitaev, J. Remmel and E. Steingr\'imsson, Enumerating $({\bf2+2})$-free posets by indistinguishable elements, J. Comb., {\bf2} (2011), 139--163.


\bibitem{du} D. Dumont, Interpr\'etations combinatoires des numbers de Genocchi (in French), Duke Math. J., {\bf41} (1974), 305--318.


\bibitem{fi1}
P.C. Fishburn, Intransitive indifference with unequal indifference intervals,
 J. Math. Psych.,  {\bf7(1)} (1970), 144--149.

\bibitem{fi2}
P.C. Fishburn, {\em Interval orders and interval graphs:  a study of
partially ordered sets},
John Wiley $\&$ Sons, 1985.

\bibitem{fo}D. Foata, Distributions eul\'eriennes et mahoniennes sur le groupe des permutations. (French) With a comment by Richard P. Stanley. NATO Adv. Study Inst. Ser., Ser. C: Math. Phys. Sci., 31, Higher combinatorics, pp. 27--49, Reidel, Dordrecht-Boston, Mass., 1977.

\bibitem{je}V. Jel\'inek, Counting general and self-dual interval orders, J. Combin. Theory Ser. A, {\bf119} (2012), 599--614.

\bibitem{js} E.Y. Jin and M.J. Schlosser, On a bi-symmetric quadruple equidistribution of a double Euler--Stirling statistics on ascent sequences, in preparation.

\bibitem{je2} V. Jel\'inek, Catalan pairs and Fishburn triples, Adv. in Appl. Math., {\bf70} (2015), 1--31.

\bibitem{kl2} D. Kim and Z. Lin, Refined restricted inversion sequences (extended abstract at FPSAC 2017),  S\'em. Lothar. Combin., {\bf78B} (2017), Art. 52, 12pp.

\bibitem{kr2}S. Kitaev and J. Remmel, Enumerating $({\bf2+2})$-free posets by the number of minimal elements and other statistics, Discrete Appl. Math., 159,  {\bf17} (2011), 2098--2108.

\bibitem{kr}S. Kitaev and J. Remmel, A note on $p$-ascent sequences, J. Comb., {\bf3} (2017), 487--506.

\bibitem{lev}P. Levande, Fishburn diagrams, Fishburn numbers and their refined generating functions, J. Combin. Theory Ser. A, {\bf120} (2013), 194--217.

\bibitem{leh} D.H. Lehmer, Teaching combinatorial tricks to a computer, in: Proc. Sympos. Appl. Math., Vol. 10, Amer. Math. Soc, Providence, RI, 1960, pp. 179--193.

\bibitem{kl} Z. Lin and D. Kim, A sextuple equidistribution arising in Pattern Avoidance, J. Combin. Theory Ser. A, 155 (2018), 267--286.

\bibitem{oeis} OEIS Foundation Inc., The On-Line Encyclopedia of Integer Sequences,  \href{http://oeis.org}{http://oeis.org}, 2011.


\bibitem{sto}
A. Stoimenow, Enumeration of chord diagrams and an upper bound for Vassiliev invariants,
J. Knot Theory Ramifications, {\bf7} (1998), 93--114.

\bibitem{yan} S.H.F. Yan, On a conjecture about enumerating $({\bf2+2})$-free posets, European J. Combin., {\bf32} (2011), 282--287.

\bibitem{zag}
D. Zagier, Vassiliev invariants and a strange identity related to the Dedekind
eta-function, Topology,  {\bf40} (2001), 945--960.







\end{thebibliography}
\end{document}